\documentclass[11pt,a4paper]{article}
\usepackage[a4paper,margin=30mm]{geometry}

\usepackage[T1]{fontenc}
\usepackage[utf8]{inputenc}

\usepackage{amsmath,amsthm,amssymb,bm}
\usepackage{mathtools}
\usepackage{mathrsfs}
\usepackage{cases}

\usepackage{graphicx,xcolor}
\usepackage{float}

\usepackage{enumitem}

\usepackage[colorlinks,urlcolor=blue,setpagesize=false]{hyperref}
\usepackage{orcidlink}

\theoremstyle{plain}
\newtheorem{thm}{Theorem}[section]
\newtheorem{lem}[thm]{Lemma}

\newtheorem{prop}[thm]{Proposition}

\theoremstyle{definition}
\newtheorem{df}[thm]{Definition}

\theoremstyle{remark}
\newtheorem{rem}[thm]{Remark}

\newcommand{\relmiddle}[1]{\mathrel{}\middle#1\mathrel{}}
\numberwithin{equation}{section}
\allowdisplaybreaks[3]

\newcommand{\E}{\mathbb{E}}

\renewcommand{\L}{\mathbb{L}}
\newcommand{\N}{\mathbb{N}}
\renewcommand{\P}{\mathbb{P}}

\newcommand{\R}{\mathbb{R}}
\renewcommand{\S}{\mathbb{S}}

\newcommand{\Z}{\mathbb{Z}}
\newcommand{\1}{\mbox{\rm1}\hspace{-0.25em}\mbox{\rm l}}

\newcommand{\cA}{\mathcal{A}}

\newcommand{\cC}{\mathcal{C}}

\newcommand{\cE}{\mathcal{E}}
\newcommand{\cF}{\mathcal{F}}

\newcommand{\cH}{\mathcal{H}}

\newcommand{\cK}{\mathcal{K}}

\newcommand{\cP}{\mathcal{P}}
\newcommand{\cQ}{\mathcal{Q}}

\newcommand{\cS}{\mathcal{S}}
\newcommand{\cT}{\mathcal{T}}

\newcommand{\scrS}{\mathscr{S}}

\renewcommand{\a}{\alpha}
\renewcommand{\b}{\beta}
\newcommand{\gm}{\gamma}
\newcommand{\Gm}{\Gamma}
\newcommand{\dl}{\delta}
\newcommand{\Dl}{\Delta}
\newcommand{\eps}{\varepsilon}

\newcommand{\Lm}{\Lambda}
\newcommand{\sg}{\sigma}
\newcommand{\Sg}{\Sigma}

\newcommand{\Om}{\Omega}

\newcommand{\del}{\partial}

\newcommand{\bond}{\mathrm{bond}}
\newcommand{\core}{\mathrm{core}}
\newcommand{\cycle}{\mathrm{cycle}}
\newcommand{\face}{\mathrm{face}}
\newcommand{\fin}{\mathrm{fin}}
\newcommand{\shield}{\mathrm{shield}}

\newcommand{\tube}{\mathrm{tube}}

\DeclareMathOperator{\diam}{diam}
\DeclareMathOperator{\Ends}{Ends}
\DeclareMathOperator{\im}{Im}
\DeclareMathOperator{\Inf}{Inf}
\DeclareMathOperator{\Var}{Var}
\DeclareMathOperator{\width}{width}

\usepackage{comment}

\usepackage[colorinlistoftodos,prependcaption,textsize=tiny]{todonotes}

\title{
Sharp Phase Transition for the Formation of Infinite Tubes
}

\author{
Shu \textsc{Kanazawa}$^{1}$\orcidlink{0000-0002-3368-1629}
\and
Omer \textsc{Bobrowski}$^{1}$\orcidlink{0000-0002-0860-7099}
\and
Primoz \textsc{Skraba}$^{1}$\orcidlink{0000-0002-5300-8984}
}

\date{}

\begin{document}
\maketitle

\begin{center}
$^{1}$ School of Mathematical Sciences, Queen Mary University of London, London, UK.\\[2mm]

E-mail: s.kanazawa@qmul.ac.uk,
o.bobrowski@qmul.ac.uk,
p.skraba@qmul.ac.uk.\\[1mm]
\end{center}

\begin{abstract}
Classical bond percolation theory studies the conditions for a given point in a random graph to be connected to infinity, or ``escape'' to infinity, via a sequence of random edges.
In this work, we present a higher-dimensional generalization of this question, asking whether a fixed loop (or, more generally, a topological sphere) can escape to infinity via a tube formed by random plaquettes.
We refer to this phenomenon as \emph{tube percolation}.

We first compare tube percolation with previously studied higher-dimensional percolation phenomena, including face and cycle percolation.
For tubes of codimension one, we further relate the critical probability for tube percolation to those for percolation of finite clusters and shielded percolation in the dual bond percolation model.

Next, we introduce a tubular analogue of the classical one-arm event, the \emph{tubular one-arm event}, and prove that it exhibits a sharp threshold at criticality: below criticality, its probability decays exponentially in scale, whereas above criticality, it admits a mean-field-type lower bound.
The proof relies on the O'Donnell--Saks--Schramm--Servedio~(OSSS) inequality together with an exploration algorithm adapted to the topology of tubes.

Finally, we study the \emph{tubular box-crossing property}.
Unlike ordinary path connectedness, ``tube connectedness'' is not transitive, and thus there is no natural notion of clusters.
Nevertheless, we establish an analogue of the uniqueness of the infinite cluster from classical bond percolation.
Combining this result with the sharp threshold for the tubular one-arm event, we prove that the existence of a box-crossing tube also exhibits a sharp threshold.
\end{abstract}

\noindent\textbf{Acknowledgements.}
The authors would like to thank Stephen Muirhead for valuable discussions on related topics in percolation theory.
OB and SK were partially supported by the EPSRC grant EP/Y008642/1.
OB and PS were partially supported by the EPSRC grant EP/Y028872/1, and by the Leverhulme Trust grant RPG-2023-144.

\section{Introduction}
\paragraph{Background.}
Percolation theory studies the formation of large-scale structures in random media.
The subject was originally introduced by Broadbent and Hammersley~\cite{BH1957} in the context of fluid flow through porous materials, and has since developed into a central topic in probability theory and statistical physics.
In its classical formulation, one considers the Bernoulli bond percolation model on the $d$-dimensional hypercubic lattice, where each edge is declared open (present) independently with probability $p$.
A fundamental question is whether an infinite cluster emerges in the random subgraph consisting of open edges.
It is well known that this phenomenon exhibits a phase transition: there exists a critical probability $p_c\in(0,1)$ such that an infinite open cluster exists almost surely when $p>p_c$, whereas it almost surely does not exist when $p<p_c$.
The properties of this phase transition, as well as the geometric structure of open clusters, have been investigated extensively.
We refer the reader to Grimmett's monograph~\cite{Gr1999} and the survey of Duminil-Copin~\cite{D2018} for comprehensive introductions to the subject.

\smallskip
There has been considerable interest in higher-dimensional analogues of percolation models and phenomena.
One natural extension of bond percolation replaces edges by $k$-dimensional cubical cells~(\emph{$k$-plaquettes}) for a fixed $1\le k\le d-1$.
This leads to the \emph{$k$-plaquette percolation model} on $\R^d$, in which each $k$-plaquette is declared open independently with probability $p$.
The $k=1$ case coincides with the classical bond percolation model.
Within this framework, one may consider the emergence of an infinite cluster of adjacent open $k$-plaquettes, a phenomenon known as \emph{face percolation}.
From this perspective, however, the higher-dimensional nature of the plaquettes does not play an essential role.
A different direction was initiated by Aizenman, Chayes, Chayes, Fr\"ohlich, and Russo~\cite{ACCFR1983}, who studied random surfaces in $\R^3$ formed by open $2$-plaquettes having a prescribed large rectangular loop as their boundary.
They proved the existence of a phase transition between the so-called \emph{area law} and \emph{perimeter law}.
Roughly speaking, the probability that such a random surface exists decays exponentially, with an exponent proportional either to the area enclosed by the loop or to its perimeter, depending on the regime.
This result was recently extended to the codimension-one case~($k=d-1$) by Duncan and Schweinhart~\cite{DS2025b}.
Their proof relies on a characterization of the existence of such hypersurfaces via a dual bond percolation representation, in which each $(d-1)$-plaquette corresponds to a unique dual edge intersecting it at its midpoint, and the dual edge is declared open if and only if the corresponding plaquette is closed.

\smallskip
In the codimension-one case, the dual bond percolation model allows many questions about plaquette configurations to be related to properties of the dual bond configuration.
From this perspective, the $(d-1)$-plaquette percolation model has been studied extensively.
Grimmett and Holroyd~\cite{GH2010} investigated the existence of hyperspheres of open $(d-1)$-plaquettes surrounding the origin in $\R^d$.
Subsequently, Grimmett, Holroyd, and Kozma~\cite{GHK2014} studied the relationship between infinite hypersurfaces of open $(d-1)$-plaquettes~(also referred to as \emph{cycle percolation}) and macroscopic structures formed by finite open clusters in the dual bond percolation model.
In particular, they showed that in sufficiently high dimensions an infinite hypersurface of open $(d-1)$-plaquettes may coexist with an infinite component consisting of neighboring finite open clusters in the dual model.
From a more algebraic-topological perspective, Hiraoka and Mikami~\cite{HM2020} investigated the structure of $(d-1)$-dimensional homology generators arising in the $(d-1)$-plaquette percolation model, which naturally correspond to finite open clusters in the dual model.
In contrast, for general $k$-plaquette percolation models, systematic techniques for their study remain largely unexplored.

\paragraph{Our contribution.}
The objective of this paper is to study a new higher-dimensional percolation phenomenon in the general $k$-plaquette percolation model~($1\le k\le d-1$).
While classical bond percolation asks whether a point (a zero-dimensional object) can escape to infinity through open edges, in the present work we ask whether a loop (or, more generally, a topological sphere) can escape to infinity via a tube formed by open plaquettes.
We therefore study the emergence of infinite $k$-dimensional tubes in the $k$-plaquette percolation model in $\R^d$.

\smallskip
We state our main results informally here; precise statements will be given in Section~\ref{sec:main}, after the necessary definitions and notation have been introduced in Section~\ref{sec:prelim}.
\begin{enumerate}
  \item We prove that there exists a critical probability $p_c^{\tube}\in(0,1)$ such that tube percolation occurs almost surely when $p>p_c^{\tube}$, and it almost surely does not occur when $p<p_c^{\tube}$.
  In Theorem~\ref{thm:comparison}, we compare $p_c^{\tube}$ with the critical probabilities associated with several other higher-dimensional percolation phenomena, such as face, core, and cycle percolation.
  In the codimension-one case~($k=d-1$), we further relate $p_c^{\tube}$ to the critical probabilities for percolation of finite open clusters and shielded percolation in the dual bond percolation model.
  \item We analyze finite-volume approximations of tube percolation.
  In classical bond percolation, the properties of the phase transition are studied via the behavior of \emph{one-arm events}, that is, the event that the origin is connected to the boundary of the box $\Lm_n:=[-n,n]^d$ by an open path.
  In the present setting, we introduce an analogue of this event, which we call the \emph{tubular one-arm event}.
  For example, in the case $k=2$, this is the event that a loop near the origin is attached to a tube that reaches the boundary of $\Lm_n$.
  We show that tube percolation can be approximated by such events~(Proposition~\ref{prop:approx}) and that the tubular one-arm probability exhibits a sharp threshold at $p_c^{\tube}$~(Theorem~\ref{thm:sharp_1-arm}).
  More precisely, below $p_c^{\tube}$ the tubular one-arm probability decays exponentially in $n$, and above $p_c^{\tube}$ it admits a uniform lower bound proportional to $p-p_c^{\tube}$.
  \item We investigate an analogue of box-crossing events for tubes.
  In analogy with classical (left-to-right) box-crossing events for open paths, we consider the event that an open $k$-dimensional tube of bounded width connects two opposite faces of a large box.
  We prove that this \emph{tubular box-crossing property} also exhibits a sharp threshold at $p_c^{\tube}$~(Theorem~\ref{thm:box-crossing}): below $p_c^{\tube}$ the crossing probability tends to zero as the box size grows, whereas above $p_c^{\tube}$ it tends to one.
  This provides another finite-volume characterization of the phase transition for tube percolation.
  In analyzing this property, we establish an analogue of the uniqueness theorem for the infinite open cluster in classical bond percolation~(Theorem~\ref{thm:uniqueness}).
  Roughly speaking, almost surely, any pair of infinite tubes can be glued together into a single bi-infinite tube.
  This uniqueness-type result enables us to combine two large tubes that appear in a large box from the left and right faces, respectively, into a single tube connecting the two opposite faces of the box.
\end{enumerate}

In contrast to classical bond percolation in which open paths can be concatenated in a straightforward manner, tubes do not admit such a simple concatenation procedure, as the tube structure may be destroyed at the junctions.
As a result, relating the existence of infinite tubes to tubular one-arm events is substantially more delicate than in the classical setting.
To address this difficulty, we introduce the notion of ``collars,'' which are loops (or spheres) together with a local trajectory that passes through them.
This additional structure enables us to control how tubes are ``glued'' together and plays a crucial role in both the approximation result mentioned in item~(2) and the uniqueness-type result in item~(3).

\smallskip
To prove the sharp threshold for the tubular one-arm event stated in item~(2), we employ the O'Donnell--Saks--Schramm--Servedio~(OSSS) inequality~\cite{OSSS2005}.
A central step is the design of an appropriate randomized algorithm that determines the tubular one-arm event while keeping the revealment probability of individual plaquettes sufficiently small.
The topology of tubes makes this task more involved than in classical path-connection cases.
Instead of tracking vertices and growing clusters, the algorithm we employ tracks loops (or spheres) and the growth of tubes around them.

\paragraph{Related work.}
Several related higher-dimensional percolation models have been investigated from different perspectives.
Hiraoka and Shirai~\cite{HS2016} introduced a higher-dimensional generalization of the classical random-cluster model~\cite{FK1972,Gr2004}.
In this model, configurations of open plaquettes are sampled according to a probability measure whose weights depend not only on the number of open plaquettes but also on the topology of the resulting configuration, thereby introducing long-range dependencies among plaquettes.
This model is closely related to higher-dimensional Potts lattice gauge theory, and its properties have been studied extensively in recent work by Duncan and Schweinhart~\cite{DS2025b,DS2025a,DS2025c}.

\smallskip
Another direction concerns continuum models.
In the setting of continuum percolation, Hirsch and Valesin~\cite{HV2025} studied analogues of face and cycle percolation in random geometric complexes.
In particular, their work focuses on Vietoris--Rips complexes built over a stationary Poisson process in Euclidean space.
In this setting, one may also consider the emergence of an infinite cluster of adjacent simplices.
They established a sharp phase transition for face percolation and derived comparison results between the critical intensities for face and cycle percolation.

\smallskip
A further related line of research is the study of \emph{homological percolation}, introduced by Bobrowski and Skraba~\cite{BS2020,BS2022}.
In this framework, one considers random media~(e.g., the Poisson--Boolean model and Gaussian random fields) on a flat torus and investigates the formation of large-scale homological features.
Here, ``large-scale'' is not defined in terms of geometric size, but rather in terms of the topology of the ambient torus.
More precisely, a cycle in the random medium is called a \emph{giant cycle} if it represents a nontrivial homology class of the ambient torus.
Intuitively, such cycles correspond to structures that wrap around the torus rather than remaining confined to a local region.
Bobrowski and Skraba~\cite{BS2022} showed that giant $k$-cycles~(for all $1\le k\le d-1$) appear in the thermodynamic regime of continuum percolation.
Subsequently, variants of this framework have also been studied in~\cite{DKS2025,DS2025a,SS2026}.

\paragraph{Outline.}
The remainder of the paper is organized as follows.
Section~\ref{sec:prelim} reviews the plaquette percolation model and several relevant percolation phenomena.
Section~\ref{sec:main} introduces tube percolation and states our main results.
Section~\ref{sec:critical} establishes comparisons between tube percolation and other percolation phenomena introduced in the preliminaries.
Section~\ref{sec:approx} proves that the tube percolation event can be approximated by tubular one-arm events.
Section~\ref{sec:sharp_1-arm} develops a randomized algorithm and applies the OSSS inequality to prove the sharp threshold for the tubular one-arm event.
Section~\ref{sec:uniqueness} discusses an analogue of the uniqueness of the infinite open cluster.
Section~\ref{sec:box-crossing} establishes the sharp threshold for the tubular box-crossing property using this analogue of the uniqueness theorem.
Finally, Section~\ref{sec:discussion} discusses several open problems and possible directions for future research in higher-dimensional percolation.

\section{Preliminaries}
\label{sec:prelim}
Throughout this article, we fix an integer $d\ge2$ as the dimension of the ambient space $\R^d$, on which models and phenomena are considered.
In this section, we review the plaquette percolation model on $\R^d$ and relevant percolation phenomena.

\subsection{Plaquette Percolation Model}
An \textit{elementary interval} is a closed interval $I\subset\R$ of the form either $I=[l,l+1]$ or $I=\{l\}$ for some integer $l\in\Z$.
In the former case it is called \emph{nondegenerate}, and in the latter case \emph{degenerate}.
An \textit{elementary cube} in $\R^d$ is a product $I_1\times I_2\times\cdots\times I_d\subset\R^d$, where each $I_i$ is an elementary interval.
A \textit{cubical set} in $\R^d$ is a (possibly infinite) union of elementary cubes in $\R^d$.
Fix an integer $1\le k\le d-1$.
An elementary cube consisting of exactly $k$ nondegenerate and $d-k$ degenerate elementary intervals is called a \textit{$k$-plaquette} in $\R^d$.
Let $\cK_k^d$ denote the set of all $k$-plaquettes in $\R^d$.
For a subset $\Lm\subset\R^d$, we also write $\cK_k^d(\Lm):=\{Q\in\cK_k^d\mid Q\subset\Lm\}$.
The \emph{$k$-dimensional complete cubical set} $|\cK_k^d|$ is defined by
\[
|\cK_k^d|:=
\bigcup_{Q\in\cK_k^d}Q.
\]

\smallskip
Let $\Om:=\{0,1\}^{\cK_k^d}$, equipped with the product $\sg$-field $\cF$.
An element $\eta\in\Om$ is called a \emph{configuration}.
For each configuration $\eta\in\Om$, we define its \emph{geometric realization} by
\[
|\eta|:=|\cK_{k-1}^d|\,\cup\,\bigcup_{Q\in\cK_k^d\colon\eta(Q)=1}Q.
\]
Then $|\eta|$ is a cubical set.
In particular, $|\eta|=|\cK_{k-1}^d|$ if $\eta\equiv0$, and $|\eta|=|\cK_k^d|$ if $\eta\equiv1$.
For $p\in[0,1]$, let $\P_p$ denote the product measure on $(\Om,\cF)$ with parameter $p$, that is, $\P_p(\eta(Q)=1)=p$ for each $k$-plaquette $Q$, and let $\E_p[\cdot]$ be the associated expectation.
Given $\eta\in\Om$, a $k$-plaquette $Q$ is said to be \emph{open} if $\eta(Q)=1$, and \emph{closed} otherwise.
The probability space $(\Om,\cF,\P_p)$ is referred to as the \emph{$k$-plaquette percolation model} on $\R^d$.
When $k=1$, this model reduces to classical bond percolation model on the $d$-dimensional hypercubic lattice $\L^d=(\Z^d,\E^d)$, where $\E^d:=\{\{x,y\}\subset\Z^d\mid\|x-y\|_1=1\}$.
The case $k=d-1$ has been studied in~\cite{ACCFR1983,GH2010,GHK2014,HM2020}.

\subsection{Face, Core, and Cycle Percolation}
We briefly review several other higher-dimensional percolation phenomena for comparison with our tube percolation.
Let $1\le k\le d-1$.

\paragraph{Face percolation.}
Two distinct $k$-plaquettes $Q,Q'\in\cK_k^d$ are said to be \emph{face-adjacent} if they share a common $(k-1)$-plaquette.
A collection $\cQ\subset\cK_k^d$ of $k$-plaquettes is said to be \emph{face-connected} if, for every pair $Q,Q'\in\cQ$, there exits a finite sequence $(Q_0,Q_1,\ldots,Q_I)$ of $k$-plaquettes in $\cQ$ such that $Q_0=Q$, $Q_I=Q'$, and $Q_{i-1}$ and $Q_i$ are face-adjacent for all $i=1,2,\ldots,I$.
We say that \emph{face percolation} occurs if there exists an infinite face-connected collection consisting of open $k$-plaquettes.
This can be viewed as site percolation on the graph whose vertices are $k$-plaquettes, with edges connecting face-adjacent pairs.
The \emph{critical probability} is defined by
\[
p_c^{\face}(k,d):=\inf\{p\in[0,1]\mid\P_p(\text{face percolation occurs})>0\}.
\]
When $k=1$, the critical probability $p_c^{\face}(1,d)$ coincides with the critical probability $p_c^{\bond}(d)$ for bond percolation on the hypercubic lattice $\L^d$.

\paragraph{Core percolation.}
A collection $\cQ\subset\cK_k^d$ of $k$-plaquettes is called a \emph{core} if every $(k-1)$-plaquette contained in the cubical set $\bigcup_{Q\in\cQ}Q\subset\R^d$ is incident to at least two $k$-plaquettes in $\cQ$.
In other words, $\bigcup_{Q\in\cQ}Q$ contains no $(k-1)$-plaquettes of degree one and hence admits no elementary collapse.
We say that \emph{core percolation} occurs if there exists an infinite face-connected core consisting of open $k$-plaquettes.
The \emph{critical probability} is defined by
\[
p_c^{\core}(k,d):=\inf\{p\in[0,1]\mid\P_p(\text{core percolation occurs})>0\}.
\]

\paragraph{Cycle percolation.}
For a collection $\cQ\subset\cK_k^d$ of $k$-plaquettes, define its (set-theoretic) \emph{boundary} by
\[
\partial_s\cQ
:=\{P\in\cK_{k-1}^d\mid\text{$\#\{Q\in\cQ\mid Q\supset P\}$ is odd}\}.
\]
We call a collection $\cQ$ a \emph{cycle} if $\partial_s\cQ=\emptyset$.
In homological terms, this is a $k$-cycle with $(\Z/2\Z)$-coefficients in the sense of Borel--Moore homology, since the collection $\cQ$ is allowed to be infinite.
We say that \emph{cycle percolation} occurs if there exists an infinite face-connected cycle consisting of open $k$-plaquettes.
The \emph{critical probability} is defined by
\[
p_c^{\cycle}(k,d):=\inf\{p\in[0,1]\mid\P_p(\text{cycle percolation occurs})>0\}.
\]

\begin{rem}
  Whenever necessary, we work with the completion of the probability space $(\Om,\cF,\P_p)$.
  Accordingly, all percolation events considered in this article are measurable with respect to the completed $\sg$-field, since they can be realized as analytic sets via a standard projection argument~(see, e.g.,~\cite{Bo2007}).
\end{rem}

By definition, it follows that
\[
p_c^{\face}(k,d)\le p_c^{\core}(k,d)\le p_c^{\cycle}(k,d).
\]
When $k=1$, we have
\begin{equation}\label{eq:all-bond_1}
p_c^{\face}(1,d)=p_c^{\core}(1,d)=p_c^{\cycle}(1,d)=p_c^{\bond}(d).
\end{equation}
Indeed, for $p>p_c^{\bond}(d)$, with positive probability, there exist two infinite self-avoiding edge-disjoint paths starting at the origin, which together form a $1$-cycle.
Moreover, it follows from a classical path-counting argument that $p_c^{\face}(k,d)>0$~(see~\cite[Theorem~1.33]{Gr1999}; see also~\cite[Proposition~2.1]{HM2020}).
As we will see in the next section, it also holds that $p_c^{\cycle}(k,d)<1$.

\subsection{Percolation of Finite Clusters and Shielded Percolation}
We review two percolation phenomena in the bond percolation model on $\L^d$ (that is, the $1$-plaquette percolation model on $\R^d$) that are closely related to cycle percolation.

\paragraph{Percolation of finite clusters.}
Let $\eta$ be a configuration of the bond percolation model on $\L^d$, and let $|\eta|$ denote the corresponding open subgraph of $\L^d$.
Two distinct finite open clusters (i.e., finite connected components of $|\eta|$) $C$ and $C'$ are said to be \emph{adjecent} if
\[
d_1(C,C'):=\inf\{\|x-y\|_1\mid x\in C,\,y\in C'\}=1.
\]
This adjacency relation induces an (abstract) graph, denoted by $G^{\fin}(\eta)$, whose vertices are the finite open clusters.
We say that \emph{percolation of finite clusters} occurs if the graph $G^{\fin}(\eta)$ contains an infinite connected component.
When $p$ is small, the percolation of finite clusters occurs with positive probability; however, for sufficiently large $p$, this phenomenon ceases.
The \emph{critical probability} is defined by
\[
p_c^{\fin}(d):=\sup\{p\in[0,1]\mid\P_p(\text{percolation of finite clusters occurs})>0\}.
\]
Since all open clusters are almost surely finite for $p<p_c^{\bond}(d)$, it follows that
\[
p_c^{\bond}(d)\le p_c^{\fin}(d).
\]
In~\cite{GHK2014}, using the triangle condition~\cite{AN1984,BA1991}, it was proved that the strict inequality $p_c^{\bond}(d)<p_c^{\fin}(d)$ holds whenever $d\ge19$.
In other words, there exists a regime of $p$ for which both the infinite open cluster and an infinite component of neighboring finite open clusters coexist.
Subsequently, the triangle condition was verified for $d\ge11$ in~\cite{FH2017}, and hence the strict inequality $p_c^{\bond}(d)<p_c^{\fin}(d)$ holds for $d\ge11$.
In~\cite{BDNS2020}, this was further improved to $d\ge10$ via comparison with \emph{shielded percolation}~(see below).

In~\cite{GHK2014}, the relationship between the cycle percolation in the case $k=d-1$ and the percolation of finite clusters was established via a dual model.
More precisely, it was shown that
\begin{equation}\label{eq:cycle-fin}
p_c^{\cycle}(d-1,d)\le1-p_c^{\fin}(d).
\end{equation}
Since $p_c^{\cycle}(k,d)$ is nonincreasing in $d$ (which follows immediately by restricting the $k$-plaquette percolation model on $\R^d$ to a lower-dimensional coordinate subspace), we obtain
\[
p_c^{\cycle}(k,d)
\le p_c^{\cycle}(k,k+1)
\le1-p_c^{\fin}(k+1)
\le1-p_c^{\bond}(k+1)<1.
\]
Consequently, for any $1\le k\le d-1$,
\[
0<p_c^{\face}(k,d)\le p_c^{\core}(k,d)\le p_c^{\cycle}(k,d)<1.
\]

\paragraph{Shielded percolation.}
In the bond percolation model on $\L^d$, a vertex $x\in\Z^d$ is said to be \emph{shielded} if all edges incident to $x$ are closed.
We say that \emph{shielded percolation} occurs if the subgraph of the hypercubic lattice $\L^d$ induced by the shielded vertices contains an infinite connected component.
When $p$ is small, shielded percolation occurs with positive probability; however, for sufficiently large $p$, it no longer occurs.
The \emph{critical probability} is defined by
\[
p_c^{\shield}(d):=\sup\{p\in[0,1]\mid\P_p(\text{shielded percolation occurs})>0\}.
\]
Since each shielded vertex is a finite open cluster, it follows that
\[
p_c^{\shield}(d)\le p_c^{\fin}(d).
\]
In~\cite{BDNS2020}, it was shown that
\begin{equation}\label{eq:bond-shield}
p_c^{\bond}(d)<p_c^{\shield}(d)\quad\text{for $d\ge10$.}
\end{equation}
Consequently, the strict inequality $p_c^{\bond}(d)<p_c^{\fin}(d)$ is known to hold for all $d\ge10$.
Combining the above inequalities, we obtain
\begin{equation}\label{eq:cycle-strict}
p_c^{\cycle}(d-1,d)<1-p_c^{\bond}(d)\quad\text{for $d\ge10$.}
\end{equation}
In particular, this shows that the duality relation $p_c^{\cycle}(d-1,d)+p_c^{\cycle}(1,d)=1$ fails to hold for $d\ge10$.

\section{Main Results}
\label{sec:main}
In this section, we define \emph{tube percolation} and state the main results.
Let $1\le k\le d-1$.
In what follows, a topological space that is homeomorphic to the $(k-1)$-dimensional sphere $\S^{k-1}$ is called a \emph{topological $(k-1)$-sphere}.
In particular, a topological $1$-sphere is called a \emph{loop}.
We first introduce the notion of \emph{tubes}.
\begin{df}[$k$-tube]
  Let $X\subset\R^d$.
  A continuous map $\gm\colon\S^{k-1}\times[a,b]\to X$~(with $-\infty<a<b<\infty$) is called a \emph{$k$-tube} in $X$ if $\gm$ is \emph{locally self-avoiding}; that is, for every $t\in[a,b]$, there exists a neighborhood $U\subset[a,b]$ of $t$ such that the restriction $\gm|_{\S^{k-1}\times U}$ is injective.
  If $\gm$ itself is injective, then $\gm$ is said to be \emph{self-avoiding}.
  The topological $(k-1)$-spheres $S=\gm(\S^{k-1}\times\{a\})$ and $S'=\gm(\S^{k-1}\times\{b\})$ are called, respectively, the \emph{initial sphere} and the \emph{terminal sphere} of $\gm$ (the \emph{initial loop} and the \emph{terminal loop} of $\gm$ when $k=2$).
  In this case, we further say that $S$ and $S'$ are \emph{joined} in $X$, and denote this by $S\overset X{\approx}S'$.
  (See Figure~\ref{fig:tube}.)
\end{df}

\begin{figure}[H]
  \centering
  \begin{minipage}[t]{0.3\textwidth}
    \centering
    \includegraphics[width=4.51cm]{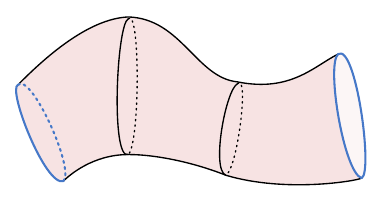}
    {\small (a)}
  \end{minipage}\hfill
  \begin{minipage}[t]{0.3\textwidth}
    \centering
    \includegraphics[width=4.51cm]{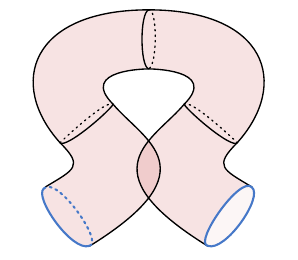}
    {\small (b)}
  \end{minipage}\hfill
  \begin{minipage}[t]{0.3\textwidth}
    \centering
    \includegraphics[width=4.51cm]{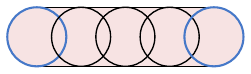}
    {\small (c)}
  \end{minipage}

  \caption{Illustration of continuous maps $\gm\colon\S^1\times[a,b]\to\R^3$.
  The blue loops in each panel represent $\gm(\S^1\times\{a\})$ and $\gm(\S^1\times\{b\})$.
  (a)~A self-avoiding $2$-tube.
  (b)~A $2$-tube that is not self-avoiding, as it has a self-intersection in the middle.
  (c)~A loop sliding horizontally.
  This is not a $2$-tube, as local self-avoidance is violated.}
  \label{fig:tube}
\end{figure}

We define the \emph{width} of a $k$-tube $\gm$ by
\[
\width(\gm):=\sup_{t\in[a,b]}\diam\gm(\S^{k-1}\times\{t\})<\infty.
\]
Here, the diameter $\diam(\cdot)$ is taken with respect to the $\ell^\infty$-norm on $\R^d$; the particular choice of norm does not affect the results qualitatively.
In the case $k=2$, informally speaking, the width measures the maximal stretching of a loop as it traverses the tube from the initial loop to the terminal loop.
We next define infinite $k$-tubes.
\begin{df}[Infinite $k$-tube]\label{df:inf_tube}
  Let $X\subset\R^d$.
  A continuous map $\gm\colon\S^{k-1}\times[0,\infty)\to X$ is called an \emph{infinite $k$-tube} in $X$ if the following conditions hold:
  \begin{itemize}
    \item (\emph{Local self-avoidance}).
    For every $t\in[0,\infty)$, there exists a neighborhood $U\subset[0,\infty)$ of $t$ such that the restriction $\gm|_{\S^{k-1}\times U}$ is injective.
    \item (\emph{Finite width}).
    The \emph{width}
    \[
    \width(\gm):=\sup_{t\in[0,\infty)}\diam\gm(\S^{k-1}\times\{t\})<\infty.
    \]
    \item (\emph{Escape to infinity}).
    It holds that
    \[
    \lim_{t\to\infty}d_\infty(o,\gm(\S^{k-1}\times\{t\}))=\infty,
    \]
    where $d_\infty(o,S):=\inf_{x\in S}\|x\|_\infty$ denotes the $\ell^\infty$-distance from the origin $o\in\R^d$ to $S$.
  \end{itemize}
  If $\gm$ itself is injective, then $\gm$ is said to be \emph{self-avoiding}.
  The topological $(k-1)$-sphere $S=\gm(\S^{k-1}\times\{0\})$ is called the \emph{initial sphere} of $\gm$ (the \emph{initial loop} of $\gm$ when $k=2$).
  In this case, we further say that $S$ is \emph{joined to infinity} in $X$, and denote this by $S\overset X{\approx}\infty$.
\end{df}

\medskip
We now define \emph{tube percolation} in the $k$-plaquette percolation model on $\R^d$, defined in Section~\ref{sec:prelim}.
In the present setting, we focus on the existence of an infinite $k$-tube in a configuration.
Let $\Lm_r:=[-r,r]^d\subset\R^d$ for $r\ge0$, and set $\Lm_o:=\Lm_{1/2}$.
\begin{df}[Tube percolation probability]
  Let $p\in[0,1]$.
  We define the \emph{tube percolation probability} $\theta^{\tube}(p)$ by
  \begin{equation}\label{eq:theta_tube}
  \theta^{\tube}(p)
 :=\P_p(\text{$\exists$ an infinite open $k$-tube whose initial sphere intersects $\Lm_o$}).
  \end{equation}
  Here, an infinite \emph{open} $k$-tube refers to an infinite $k$-tube in the geometric realization $|\eta|\subset\R^d$~(i.e., $X=|\eta|$ in Definition~\ref{df:inf_tube}).
  We refer to the event in~\eqref{eq:theta_tube} as the \emph{tube percolation event}, and denote it by $\cT$.
\end{df}
\noindent
In the case $k=2$, informally speaking, $\theta^{\tube}(p)$ represents the probability that a loop near the origin can ``escape'' to infinity while undergoing only bounded stretching (in a locally self-avoiding manner).

We define the \emph{critical probability} $p_c^{\tube}=p_c^{\tube}(k,d)$ by
\[
p_c^{\tube}:=\inf\{p\in[0,1]\mid\theta^{\tube}(p)>0\}.
\]
Since
\begin{align*}
&\{\text{$\exists$ an infinite open $k$-tube}\}\\
&=\bigcup_{x\in\Z^d}\{\text{$\exists$ an infinite open $k$-tube whose initial sphere intersects $x+\Lm_o$}\},
\end{align*}
it follows from the union bound and translation invariance that, for $p<p_c^{\tube}$, almost surely there exists no infinite open $k$-tube.
On the other hand, by Kolmogorov's zero-one law, the probability that there exists an infinite open $k$-tube is either zero or one; hence, for $p>p_c^{\tube}$, almost surely there exists an infinite open $k$-tube.
Consequently, the phase transition for \emph{tube percolation} occurs at $p_c^{\tube}$:
\begin{equation}\label{eq:tube_perc}
\P_p(\text{$\exists$ an infinite open $k$-tube})=
\begin{cases}
  0     &\text{if $p<p_c^{\tube}$,}\\
  1     &\text{if $p>p_c^{\tube}.$}
\end{cases}
\end{equation}
The behavior at criticality $p=p_c^{\tube}$ currently remains open.

\subsection*{Critical Probabilities}
Our first main result establishes a comparison between tube percolation and the other relevant percolation phenomena reviewed in Section~\ref{sec:prelim} and, in particular, shows that $p_c^{\tube}$ is nontrivial.
\begin{thm}\label{thm:comparison}
  Let $1\le k\le d-1$.
  Then
  \begin{equation}\label{eq:core-tube-1}
  p_c^{\core}(k,d)
  \le p_c^{\tube}(k,d)
  \le1-p_c^{\shield}(k+1).
  \end{equation}
  Moreover, $p_c^{\tube}(1,d)=p_c^{\bond}(d)$, and
  \begin{equation}\label{eq:tube-strict}
  p_c^{\tube}(k,d)
  <1-p_c^{\bond}(k+1)\quad\text{for $k\ge9$.}
  \end{equation}
\end{thm}

\begin{rem}\label{rem:s-tube}
  At present, it remains open whether the inequality $p_c^{\cycle}(k,d)\le p_c^{\tube}(k,d)$ holds, mainly due to the fact that infinite tubes may have self-intersections; see Section~\ref{sec:critical} for a discussion of the underlying difficulties.
  One may instead introduce a modified critical probability $\widetilde p_c^{\,\,\tube}(k,d)$, defined in terms of the emergence of an infinite self-avoiding open $k$-tube.
  Since this notion imposes an additional constraint, it is immediate that $p_c^{\tube}(k,d)\le\widetilde p_c^{\,\,\tube}(k,d)$.
  With this modified definition, the inequality
  \begin{equation}\label{eq:s-tube}
  p_c^{\cycle}(k,d)
  \le\widetilde p_c^{\,\,\tube}(k,d)<1
  \end{equation}
  indeed holds.
  Moreover, in the case $k=d-1$, one has
  \begin{equation}\label{eq:s-tube_d-1}
  p_c^{\cycle}(d-1,d)
  \le1-p_c^{\fin}(d)
  \le\widetilde p_c^{\,\,\tube}(d-1,d);
  \end{equation}
  see Section~\ref{sec:critical} for further details.
  However, the requirement of self-avoiding is too restrictive for our purposes, particularly in the analysis of the sharp threshold for the existence of box-crossing tubes~(Theorem~\ref{thm:box-crossing}).
\end{rem}

\subsection*{An Analogue of the Uniqueness of the Infinite Cluster}
In classical bond percolation (the $1$-plaquette percolation model on $\R^d$), the uniqueness of the infinite cluster states that for every $p\in[0,1]$, there exists at most one infinite open cluster $\P_p$-almost surely.
This may be rephrased as follows.
For every $p\in[0,1]$, the following holds $\P_p$-almost surely:
If
\[
\gm^-\colon(-\infty,0]\to\R^d
\quad\text{and}\quad
\gm^+\colon[0,\infty)\to\R^d
\]
are backward and forward infinite open paths, respectively, that is, $\gm^-$ and $\gm^+$ are continuous maps whose images are contained in the geometric realization $|\eta|\subset\R^d$ and satisfy
\[
\lim_{t\to-\infty}\|\gm^-(t)\|_\infty=\infty
\quad\text{and}\quad
\lim_{t\to\infty}\|\gm^+(t)\|_\infty=\infty,
\]
then there exists a bi-infinite open path $\gm\colon\R\to\R^d$ such that
\[
\gm|_{(-\infty,-T]}=\gm^-|_{(-\infty,-T]}
\quad\text{and}\quad
\gm|_{[T,\infty)}=\gm^+|_{[T,\infty)}
\]
for some $T\ge0$.
In short, the tails of the infinite open paths $\gm^-$ and $\gm^+$ can be combined to form a single bi-infinite open path $\gm$.

\smallskip
We establish a tubular analogue of the above statement.
An infinite $k$-tube $\gm\colon\S^{k-1}\times[0,\infty)\to X$ as in Definition~\ref{df:inf_tube} will, when it is necessary to distinguish directions, be referred to as a \emph{forward infinite $k$-tube} in $X$.
Analogously, a \emph{backward infinite $k$-tube} in $X$ is a continuous map $\gm\colon\S^{k-1}\times(-\infty,0]\to X$ defined by replacing $[0,\infty)$ and $\lim_{t\to\infty}$ in Definition~\ref{df:inf_tube} with $(-\infty,0]$ and $\lim_{t\to-\infty}$, respectively.
When $[0,\infty)$ and $\lim_{t\to\infty}$ in Definition~\ref{df:inf_tube} are replaced with $\R$ and $\lim_{t\to\pm\infty}$, respectively, we refer to such a continuous map $\gm\colon\S^{k-1}\times\R\to X$ as a \emph{bi-infinite $k$-tube} in $X$.

The following theorem can be regarded as an analogue of the uniqueness of the infinite cluster in classical bond percolation.
\begin{thm}\label{thm:uniqueness}
  Let $1\le k\le d-1$.
  Then for every $p\in[0,1]$, the following holds $\P_p$-almost surely$:$
  If
  \[
  \gm^-\colon\S^{k-1}\times(-\infty,0]\to\R^d
  \quad\text{and}\quad
  \gm^+\colon\S^{k-1}\times[0,\infty)\to\R^d
  \]
  are backward and forward infinite open $k$-tubes, respectively, then there exists a bi-infinite open $k$-tube $\gm\colon\S^{k-1}\times\R\to\R^d$ such that
  \[
  \gm|_{\S^{k-1}\times(-\infty,-T]}=\gm^-|_{\S^{k-1}\times(-\infty,-T]}
  \quad\text{and}\quad
  \gm|_{\S^{k-1}\times[T,\infty)}=\gm^+|_{\S^{k-1}\times[T,\infty)}
  \]
  for some $T\ge0$.
\end{thm}

\subsection*{Sharp Threshold for Tubular One-Arm Events}
We next introduce finite-volume events, depending only on the states of finitely many $k$-plaquettes, that approximate the tube percolation event $\cT$.

\begin{df}[Tubular one-arm probability]
  Let $p\in[0,1]$.
  For a width parameter $w>0$ and a scale parameter $n\in\N$, we define the \emph{tubular one-arm probability} $\theta_{w,n}^{\tube}(p)$ by
  \begin{equation}\label{eq:1-arm_event}
  \theta_{w,n}^{\tube}(p)
 :=\P_p\left(
  \begin{array}{l}
    \text{$\exists$ an open $k$-tube of width at most $w$ whose initial sphere}\\
    \text{intersects $\Lm_o$ and terminal sphere intersects $\del\Lm_n$}
  \end{array}
  \right).
  \end{equation}
  By convention, when $n\le1/2+w$, we regard the event in~\eqref{eq:1-arm_event} as occurring trivially, and hence set $\theta_{w,n}^{\tube}(p):=1$.
  Indeed, when $n\le1/2+w$, there exists a single topological $(k-1)$-sphere intersecting both $\Lm_o$ and $\del\Lm_n$, which may be viewed as a trivial $k$-tube whose initial and terminal spheres coincide.
  We refer to the event in~\eqref{eq:1-arm_event} as a \emph{tubular one-arm event} with width at most $w$, and denote it by $\cT_{w,n}$~(see Figure~\ref{fig:one-arm}).
\end{df}
In the case $k=2$, informally speaking, $\theta_{w,n}^{\tube}(p)$ represents the probability that a loop near the origin can reach $\del\Lm_n$ while undergoing only limited stretching.
Since every $k$-tube is a continuous map, the monotonicity of the tubular one-arm event $\cT_{w,n}$ in $n$ follows immediately from the intermediate value theorem.
Also, the monotonicity of $\cT_{w,n}$ in $w$ is obvious by definition.
Furthermore, it is straightforward to verify that the tubular one-arm event $\cT_{w,n}$ depends only on the configuration inside $\Lm_n$.

\begin{figure}[H]
  \centering
  \includegraphics[width=7.5cm]{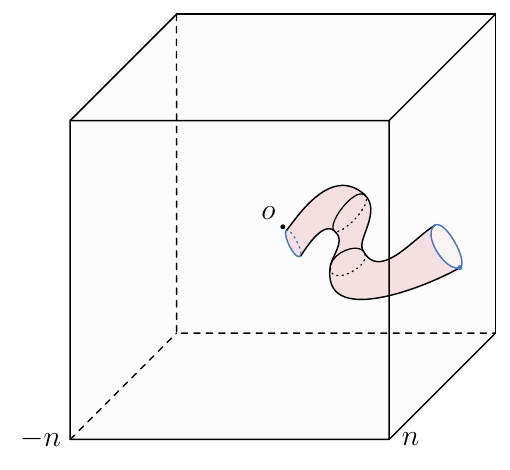}
  \caption{Illustration of the tubular one-arm event for $k=2$ and $d=3$.
  The blue loops represent the initial and terminal loops.
  The loop near the origin moves through the box $\Lm_n$ and reaches its boundary (at the blue point) while undergoing only limited stretching.}
  \label{fig:one-arm}
\end{figure}

The following proposition shows that the tubular one-arm event $\cT_{w,n}$ provides a suitable approximation to the tube percolation event $\cT$.
\begin{prop}\label{prop:approx}
  It holds that
  \begin{equation}\label{eq:approx}
  \cT=\bigcup_{w>0}\bigcap_{n\in\N}\cT_{w,n}.
  \end{equation}
  In particular,
  \[
  \theta^{\tube}(p)
  =\lim_{w\to\infty}\lim_{n\to\infty}\theta_{w,n}^{\tube}(p).
  \]
\end{prop}
\begin{rem}\label{rem:transitivity}
  Unlike in classical bond percolation, approximating the percolation event by one-arm events is considerably more delicate.
  The main difficulty arises from the fact that, in contrast to ordinary paths in bond percolation, the class of $k$-tubes is not closed under concatenation: in general, two $k$-tubes cannot be concatenated to form a single $k$-tube, since the local self-avoidance condition may be violated at the junction.
  In other words, the join relation $\approx$ fails to be transitive.
  We resolve this difficulty in Section~\ref{sec:approx} by introducing an associated directed graph on topological spheres together with a local trajectory that passes through them.
\end{rem}

\medskip
Our next result shows that the tubular one-arm event exhibits a sharp threshold at $p_c^{\tube}$: below $p_c^{\tube}$, the tubular one-arm probability $\theta_{w,n}^{\tube}(p)$ decays exponentially in $n$, whereas above $p_c^{\tube}$, the tube percolation probability $\theta^{\tube}(p)$ admits a mean-field-type lower bound.
\begin{thm}\label{thm:sharp_1-arm}
  Let $1\le k\le d-1$.
  Then the following statements hold$:$
  \begin{enumerate}
    \item If $p<p_c^{\tube}$, then for every $w>0$, there exists a constant $c_{p,w}>0$ such that
    \[
    \theta_{w,n}^{\tube}(p)\le\exp(-c_{p,w}n)
    \]
    for any integer $n>1/2+w$$;$
    \item There exists a constant $\kappa>0$ such that
    \[
    \theta^{\tube}(p)\ge\kappa(p-p_c^{\tube})
    \]
    for all $p>p_c^{\tube}$.
  \end{enumerate}
\end{thm}

\subsection*{Sharp Threshold for the Tubular Box-Crossing Property}
Our final result is a sharp threshold for the existence of a box-crossing open $k$-tube.
Let $L_n$ and $R_n$ denote the left and right faces of the box $\Lm_n$, respectively; that is,
\[
L_n:=\{-n\}\times[-n,n]^{d-1}\quad\text{and}\quad R_n:=\{n\}\times[-n,n]^{d-1}.
\]
For a width parameter $w>0$ and a scale parameter $n\in\N$, we define the \emph{tubular box-crossing event} $\cH_{w,n}$ by
\[
\cH_{w,n}:=\left\{
  \begin{array}{l}
    \text{$\exists$ an open $k$-tube in $\Lm_n$ of width at most $w$ whose initial sphere}\\
    \text{intersects $L_n$ and terminal sphere intersects $R_n$}
  \end{array}
  \right\}.
\]
The following theorem shows that the tubular box-crossing event $\cH_{w,n}$ exhibits a sharp threshold at $p_c^{\tube}$.
\begin{thm}\label{thm:box-crossing}
  Let $1\le k\le d-1$.
  Then the following statements hold$:$
  \begin{enumerate}
    \item If $p<p_c^{\tube}$, then for every $w>0$,
    \[
    \lim_{n\to\infty}\P_p(\cH_{w,n})=0;
    \]
    \item If $p>p_c^{\tube}$, then for all sufficiently large $w>0$ $($depending on $p$$)$,
    \[
    \lim_{n\to\infty}\P_p(\cH_{w,n})=1.
    \]
  \end{enumerate}
  In particular,
  \[  \lim_{w\to\infty}\lim_{n\to\infty}\P_p(\cH_{w,n})=
  \begin{cases}
    0     &\text{$p<p_c^{\tube}$,}\\
    1     &\text{$p>p_c^{\tube}$.}
  \end{cases}
  \]
\end{thm}

The first statement is an immediate consequence from the exponential decay of the tubular one-arm probability in $n$ when $p<p_c^{\tube}$~(see Theorem~\ref{thm:sharp_1-arm}(1)).
To prove the second statement, we rely on the analogue of the uniqueness of the infinite cluster in classical bond percolation~(see Theorem~\ref{thm:uniqueness}).

The remaining sections, except for Section~\ref{sec:discussion}, are devoted to the proofs of the above theorems.

\section{Critical Probabilities}
\label{sec:critical}
In this section, we prove Theorem~\ref{thm:comparison}, together with the inequalities~\eqref{eq:s-tube} and~\eqref{eq:s-tube_d-1} stated in Remark~\ref{rem:s-tube}.
We begin by establishing~\eqref{eq:s-tube}, then complete the proof of Theorem~\ref{thm:comparison}.
Finally, we verify~\eqref{eq:s-tube_d-1}.

\begin{proof}[Proof of~\eqref{eq:s-tube}]
  We first prove the first inequality: $p_c^{\cycle}(k,d)
  \le\widetilde p_c^{\,\,\tube}(k,d)$.
  Set
  \begin{align*}
  \widetilde\cT&:=\{\text{$\exists$ an infinite self-avoiding open $k$-tube whose initial sphere intersects $\Lm_o$}\}
  \shortintertext{and}
  A_n&:=\{\text{$\exists$ an infinite self-avoiding open $k$-tube whose initial sphere is contained in $\Lm_n$}\}.
  \end{align*}
  We fix $p>\widetilde p_c^{\,\,\tube}(k,d)$.
  Since $\P_p(\widetilde\cT)>0$ and $\widetilde\cT\subset\bigcup_{n\in\N}A_n$, there exists $n\in\N$ sufficiently large such that
  \begin{equation}\label{eq:An}
  \P_p(A_n)>0.
  \end{equation}
  Suppose that $\gm$ is an infinite self-avoiding open $k$-tube whose initial sphere is contained in $\Lm_n$.
  Let $\cQ_1$ denote the collection of all $k$-plaquettes contained in the image $\im(\gm)$.
  It is straightforward to verify that $\cQ_1$ is face-connected and its (set-theoretic) boundary $\cP:=\del_s\cQ_1$ is contained in $\Lm_n$.
  Therefore, there exists a face-connected collection $\cQ_2$ of $k$-plaquettes in $\Lm_n$ such that $\del_s\cQ_2=\cP$.
  (Here, we use the fact that the $(k-1)$st cubical homology group of the cubical set $\Lm_n$ is trivial, since $\Lm_n$ is contractible.)
  Consider the symmetric difference $\cQ:=\cQ_1\bigtriangleup\cQ_2$.
  Since
  \[
  \del_s\cQ
  =(\del_s\cQ_1)\bigtriangleup(\del_s\cQ_2)
  =\cP\bigtriangleup\cP
  =\emptyset,
  \]
  the collection $\cQ$ forms a cycle.
  Moreover, the collection $\cQ$ is infinite (since $\cQ_1$ is infinite) and face-connected, since both $\cQ_1$ and $\cQ_2$ are face-connected and share at least one $(k-1)$-plaquette.
  Recall that $\cK_k^d(\Lm_n)=\{Q\in\cK_k^d\mid Q\subset\Lm_n\}$.
  By the above construction,
  \begin{align*}
  \P_p(\text{cycle percolation occurs})
  &\ge\P_p(A_n\cap\{\eta(Q)=1\text{ for all $Q\in\cK_k^d(\Lm_n)$}\})\\
  &\ge\P_p(A_n)\cdot\P_p(\eta(Q)=1\text{ for all $Q\in\cK_k^d(\Lm_n)$})\\
  &>0.
  \end{align*}
  The second line follows from the FKG inequality, and the last line follows from~\eqref{eq:An}.
  Thus, $p\ge p_c^{\cycle}(k,d)$.
  Since $p>\widetilde p_c^{\,\,\tube}(k,d)$ was arbitrary, we conclude that $p_c^{\cycle}(k,d)\le\widetilde p_c^{\,\,\tube}(k,d)$.
  
  \smallskip
  We next prove the second inequality: $\widetilde p_c^{\,\,\tube}(k,d)<1$.
  It suffices to prove that $\widetilde p_c^{\,\,\tube}(k,k+1)<1$, since $\widetilde p_c^{\,\,\tube}(k,d)$ is nonincreasing in $d$ (which follows immediately by restricting the $k$-plaquette percolation model on $\R^d$ to the subspace $\R^{k+1}\times\{0\}^{d-k-1}$).
  We consider an associated site percolation on the scaled hypercubic lattice $(2\Z)^{k+1}$, whose edge set is
  \[
  \{\{x,y\}\subset(2\Z)^{k+1}\mid\|x-y\|_1=2\}.
  \]
  A vertex $x\in(2\Z)^{k+1}$ is declared \emph{occupied} if and only if every $k$-plaquette whose interior is contained in the half-open box $x+(-1,1]^{k+1}$ is open.
  An edge $\{x,y\}$ is present if and only if both endpoints $x$ and $y$ are occupied.
  Note that the probability that a given vertex is occupied is a monomial in $p$, and that the occupancies of distinct vertices are independent.
  Hence, for $p$ sufficiently close to $1$ (so that the monomial exceeds the critical probability for Bernoulli site percolation on the $(k+1)$-dimensional hypercubic lattice), with a positive probability, there exists an infinite self-avoiding path $(x_0,x_1,\ldots)$ of occupied vertices starting from the origin $o\in(2\Z)^{k+1}$.
  For each $i\ge0$, let $y_i\in\Z^{k+1}$ denote the midpoint between $x_i$ and $x_{i+1}$.
  By construction, the union
  \[
  \bigcup_{i=0}^\infty\bigl\{\bigl(x_i+[0,1]^{k+1}\bigr)\cup\bigl(y_i+[0,1]^{k+1}\bigr)\}
  \]
  contains an infinite self-avoiding open $k$-tube of width $1$ whose initial sphere contains the origin.
  This completes the proof.
\end{proof}

\begin{rem}\label{rem:core-tube}
  The above argument establishing $p_c^{\cycle}(k,d)\le\widetilde p_c^{\,\,\tube}(k,d)$ does not extend to $p_c^{\tube}(k,d)$ in place of $\widetilde p_c^{\,\,\tube}(k,d)$.
  The obstruction arises from the possibility that a $k$-tube may have self-intersection.
  In that case, for the collection $\cQ_1$ appearing in the proof, the number of incident $k$-plaquettes in $\cQ_1$ to some $(k-1)$-plaquette outside $\Lm_n$ may become odd, so that the boundary cancellation argument no longer applies.
  By contrast, in the case of core percolation, no such difficulty arises.
  Indeed, we may take the collection $\cQ_2$ in the above proof to be $\cK_k^d(\Lm_n)$, and simply define $\cQ:=\cQ_1\cup\cQ_2$.
  Then the collection $\cQ$ forms an infinite face-connected core.
  Consequently, core percolation occurs with positive probability whenever $p>p_c^{\tube}(k,d)$.
  This proves that $p_c^{\core}(k,d)\le p_c^{\tube}(k,d)$.
\end{rem}

\medskip
We next prove Theorem~\ref{thm:comparison}.
Before that, we briefly review the \emph{dual bond percolation model} of the $(d-1)$-plaquette percolation model on $\R^d$.
Let $(\L^d)^*=((\Z^d)^*,(\E^d)^*)$ denote the dual lattice, the copy of $\L^d=(\Z^d,\E^d)$ translated by the vector $(1/2,\ldots,1/2)\in\R^d$.
Given a $(d-1)$-plaquette configuration $\eta\in\Om=\{0,1\}^{\cK_{d-1}^d}$, its \emph{dual bond configuration}, denoted by $\eta^*\in\{0,1\}^{(\E^d)^*}$, is defined as follows.
Each $(d-1)$-plaquette $Q\in\cK_{d-1}^d$ corresponds to a unique dual edge $Q^*$ in the dual lattice $(\L^d)^*$ intersecting $Q$ at its midpoint.
We define $\eta^*(Q^*):=1-\eta(Q)$; that is, we declare the dual edge $Q^*$ to be open if and only if the corresponding $(d-1)$-plaquette $Q$ is closed.
Similarly, for an event $A\in\cF$, let $A^*:=\{\eta^*\mid\eta\in A\}$ denote its dual event.
Let $\Om^*:=\{0,1\}^{(\E^d)^*}$, equipped with the product $\sg$-field $\cF^*$, and define $\P_p^*(A^*):=\P_p(A)$.
Then $(\Om^*,\cF^*,\P_p^*)$ is the bond percolation model on the dual lattice $(\L^d)^*$ with open probability $1-p$.

\begin{proof}[Proof of Theorem~\ref{thm:comparison}]
  We first prove that $p_c^{\tube}(d-1,d)\le1-p_c^{\shield}(d)$.
  Fix $p>1-p_c^{\shield}(d)$.
  Then $1-p<p_c^{\shield}(d)$, and hence, with positive probability, shielded percolation occurs in the dual bond percolation model.
  Therefore, with positive probability, there exists an infinite self-avoiding path $(x_0,x_1,\ldots)$ of shielded dual vertices in the dual model, where only successive dual vertices are adjacent in $(\L^d)^*$.
  By construction, the union
  \[
  \bigcup_{i=0}^\infty(x_i+\Lm_o)
  \]
  contains an infinite open $(d-1)$-tube of width $1$.
  Consequently, $p\ge p_c^{\tube}(d-1,d)$.
  Since $p>1-p_c^{\shield}(d)$ was arbitrary, we conclude that $p_c^{\tube}(d-1,d)\le1-p_c^{\shield}(d)$.
  
  Furthermore, since $p_c^{\tube}(k,d)$ is nonincreasing in $d$, we have
  \[
  p_c^{\tube}(k,d)
  \le p_c^{\tube}(k,k+1)
  \le1-p_c^{\shield}(k+1).
  \]
  Combining this with Remark~\ref{rem:core-tube}, we obtain~\eqref{eq:core-tube-1}.
  The strict inequality~\eqref{eq:tube-strict} follows immediately from~\eqref{eq:bond-shield}.
  
  \smallskip
  Finally, we prove $p_c^{\tube}(1,d)=p_c^{\bond}(d)$.
  Fix $p>p_c^{\bond}(d)$.
  Then, with positive probability, there exists an injective Lipschitz continuous map $\gm\colon[0,\infty)\to\R^d$ such that $\im(\gm)$ is contained in the geometric realization $|\eta|\subset\R^d$ and $\gm$ satisfies
  \[
  \gm(0)=o\in\R^d
  \quad\text{and}\quad
  \lim_{t\to\infty}\|\gm(t)\|_\infty=\infty.
  \]
  We then define a continuous map $\widetilde\gm\colon\S^0\times[0,\infty)\to\R^d$ by
  \[
  \widetilde\gm(p,t):=
  \begin{cases}
    \gm(t)      &\text{if $p=-1$,}\\
    \gm(t+1)    &\text{if $p=1$.}
  \end{cases}
  \]
  It is straightforward to verify that $\widetilde\gm$ satisfies the three conditions in Definition~\ref{df:inf_tube}; hence, $\widetilde\gm$ is an infinite open $1$-tube.
  Note that although the two points in $\S^0$ travel along the same infinite open path, this does not violate the local self-avoidance condition.
  Consequently, $p\ge p_c^{\tube}(1,d)$.
  Since $p>p_c^{\bond}(d)$ was arbitrary, we conclude that $p_c^{\tube}(1,d)\le p_c^{\bond}(d)$.
  The reverse inequality $p_c^{\bond}(d)\le p_c^{\tube}(1,d)$ follows immediately from~\eqref{eq:all-bond_1} and~\eqref{eq:core-tube-1}.
  This completes the proof.
\end{proof}

\smallskip
As noted in~\eqref{eq:cycle-fin}, the first inequality in~\eqref{eq:s-tube_d-1} was established in~\cite{GHK2014}.
It therefore remains to prove the second inequality in~\eqref{eq:s-tube_d-1}, namely, $1-p_c^{\fin}(d)\le\widetilde p_c^{\,\,\tube}(d-1,d)$.
\begin{proof}[Proof of~\eqref{eq:s-tube_d-1}]
  Recall that
  \[
  \widetilde\cT:=\{\text{$\exists$ an infinite self-avoiding open $(d-1)$-tube whose initial sphere intersects $\Lm_o$}\}.
  \]
  For $w>0$, we define
  \[
  \widetilde\cT_w:=\left\{
  \begin{array}{l}
    \text{$\exists$ an infinite self-avoiding open $(d-1)$-tube of width}\\
    \text{at most $w$ whose initial sphere intersects $\Lm_o$}
  \end{array}
  \right\}.
  \]
  Fix $p>\widetilde p_c^{\,\,\tube}(d-1,d)$.
  Since $\P_p(\widetilde\cT)>0$ and $\widetilde\cT=\bigcup_{w\in\N}\widetilde\cT_w$, there exists $w\in\N$ such that
  \begin{equation}\label{eq:tilde_Tw}
  \P_p(\widetilde\cT_w)>0.
  \end{equation}
  
  \noindent\emph{Step~1: Measurable choice of a tube support.}
  For an infinite self-avoiding open $(d-1)$-tube $\gamma$, let $\cQ_\gm$ be the collection of all $(d-1)$-plaquettes contained in $\im(\gamma)$. Enumerate $\cK_{d-1}^d$ as $(Q_1,Q_2,\dots)$.
  On $\widetilde\cT_w$, among all collections $\cQ_\gm$ arising from infinite self-avoiding open tubes $\gamma$ with $\width(\gamma)\le w$ and initial sphere intersecting $\Lm_o$, choose $\widetilde\cQ=\widetilde\cQ(\eta)$ so that the indicator sequence $\bigl(\1_{\{Q_i\in\widetilde\cQ\}}\bigr)_{i\in\N}\in\{0,1\}^\N$ is lexicographically minimal.
  Outside $\widetilde\cT_w$, set $\widetilde\cQ:=\emptyset$.
  By construction, on $\widetilde\cT_w$, every $Q\in\widetilde\cQ$ is open in the primal plaquette configuration $\eta$, hence the corresponding dual edge $Q^*$ is closed in the dual bond configuration $\eta^*$.
  
  \smallskip
  \noindent\emph{Step~2: Deterministic cutsets inside the chosen tube.}
  On $\widetilde\cT_w$, using that the chosen tube is self-avoiding and has width at most $w$, we can construct two sequences
  \[
  V_1,V_2,\ldots\subset(\Z^d)^*
  \quad\text{and}\quad
  E_1,E_2,\ldots\subset(\E^d)^*
  \]
  of dual-vertex sets and dual-edge sets, respectively, such that~(see Figure~\ref{fig:fin-tube}):
  \begin{itemize}
    \item[(i)] The sets $E_i$ are pairwise disjoint and $E_i\cap\{Q^*\mid Q\in\widetilde\cQ\}=\emptyset$ for all $i$;
    \item[(ii)] $M:=\sup_{i\ge1}|E_i|<\infty$, and every $V_i$ is finite;
    \item[(iii)] The graph
    \[
    (\L^d)^*\setminus(\{Q^*\mid Q\in\widetilde\cQ\}\cup E_1)
    \]
    has exactly two infinite connected components, and let $(\L_\text{in}^d)^*$ be one of them.
    The vertex set of $(\L_\text{in}^d)^*$ coincides with the union $\bigcup_{i\ge1}V_i$.
    We call $(\L_\text{in}^d)^*$ the \emph{inside} of $\widetilde\cQ$;
    \item[(iv)] Every dual path in $(\L_\text{in}^d)^*$ that connects $V_i$ to $V_j$ ($j>i$) traverses at least one dual edge in each of the sets $E_{i+1},E_{i+2},\ldots,E_j$.
    We call each $E_i$ a \emph{cutset} for $(\L_\text{in}^d)^*$.
  \end{itemize}
  
  \noindent\emph{Step~3: No infinite open dual path inside $\widetilde\cQ$.}
  For each $i\ge1$, let
  \[
  \cE_i:=\{\text{the cutset $E_i$ contains at least one open dual edge in $\eta^*$}\}.
  \]
  Since $E_i$ is disjoint from $\{Q^*\mid Q\in\widetilde\cQ\}$ and the sets $E_i$ are pairwise disjoint, the events $(\cE_i)_{i\ge1}$ are independent conditional on $\widetilde\cQ$.
  Moreover,
  \[
  \P_p(\cE_i\mid\widetilde\cQ)
  =1-p^{|E_i|}
  \le 1-p^M
  =:a<1.
  \]
  Hence, for every $I\in\N$ and $N>I$,
  \[
  \P_p\left(\bigcap_{i=I}^N \cE_i\relmiddle|\widetilde\cQ\right)
  =\prod_{i=I}^N\P_p\Bigl(\cE_i\mid\widetilde\cQ\Bigr)
  \le a^{N-I+1}\xrightarrow[N\to\infty]{}0,
  \]
  which implies that
  \[
  \P_p\left(\bigcap_{i\ge I}\cE_i\relmiddle|\widetilde\cQ\right)=0
  \quad\text{a.s. for all $I\in\N$}.
  \]
  Therefore, by the union bound, we obtain
  \begin{equation}\label{eq:liminf_Ei}
  \P_p\left(\bigcup_{I\in\N}\bigcap_{i\ge I}\cE_i\relmiddle|\widetilde\cQ\right)=0\quad\text{a.s.}
  \end{equation}
  By Property~(iv) in Step~2, any infinite open dual path in $(\L_\text{in}^d)^*$ crosses all but finitely many cutsets, hence implies the event $\bigcup_I\bigcap_{i\ge I}\cE_i$, which has conditional probability~$0$ by~\eqref{eq:liminf_Ei}.
  Thus, conditional on $\widetilde\cQ$, every open dual cluster contained in the inside $(\L_\text{in}^d)^*$ of $\widetilde\cQ$ is finite.

  \smallskip
  \noindent\emph{Step~4: Percolation of finite clusters in the dual.}
  From Step~3, conditional on $\widetilde\cQ$, the infinite connected graph $(\L_\text{in}^d)^*$ consists entirely of finite open dual clusters, with conditional probability~$1$.
  Therefore, writing the event
  \[
  \cA:=\{\text{the graph $G^{\fin}(\eta^*)$ has an infinite connected component}\},
  \]
  we have $\P_p(\cA\mid\widetilde\cQ)=1$ a.s. on $\widetilde\cT_w$~(see Section~\ref{sec:prelim} for the definition of $G^{\fin}(\eta^*)$).
  Thus,
  \[
  \P_p(\cA)
  \ge\P_p(\cA\cap \widetilde\cT_w)
  =\E_p\big[\1_{\widetilde\cT_w}\P_p(\cA\mid\widetilde\cQ)\big]
  =\P_p(\widetilde\cT_w)>0
  \]
  by~\eqref{eq:tilde_Tw}.
  Consequently, in the dual bond percolation model with open parameter $1-p$, percolation of finite clusters occurs with positive probability, hence $1-p\le p_c^{\fin}(d)$.
  Since $p>\widetilde p_c^{\,\,\tube}(d-1,d)$ was arbitrary, we conclude that $1-p_c^{\fin}(d)\le \widetilde p_c^{\,\,\tube}(d-1,d)$.
\end{proof}

\begin{figure}[H]
  \centering
  \includegraphics[width=8.5cm]{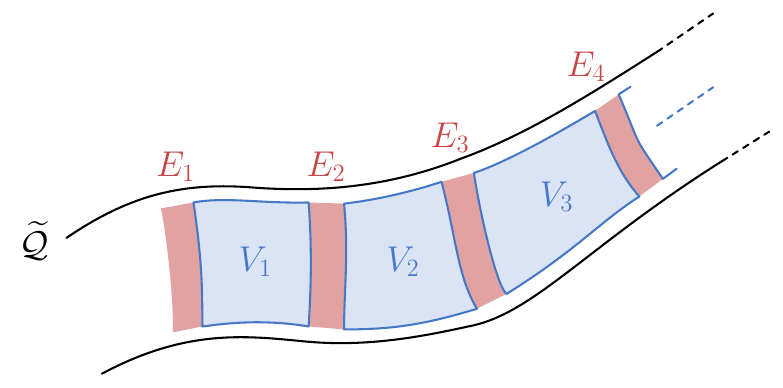}
  \caption{Illustration of the cutsets $E_i$~(shown in red) in the proof of~\eqref{eq:s-tube_d-1}.
  The two black curves represent the chosen tube support $\widetilde\cQ$.
  The inside of $\widetilde\cQ$ is given by the union of dual-vertex sets $V_i$~(shown in blue).}
  \label{fig:fin-tube}
\end{figure}

\section{Tubular One-Arm Events and Tube Percolation}
\label{sec:approx}
In this section, we prove Proposition~\ref{prop:approx}.
Recall that the tube percolation event $\cT$ and the tubular one-arm event $\cT_{w,n}$ with width at most $w$ are defined, respectively, by
\begin{align*}
\cT&:=\{\text{$\exists$ an infinite open $k$-tube whose initial sphere intersects $\Lm_o$}\}\\
\shortintertext{and}
\cT_{w,n}&:=\left\{
\begin{array}{l}
  \text{$\exists$ an open $k$-tube of width at most $w$ whose initial sphere}\\
  \text{intersects $\Lm_o$ and terminal sphere intersects $\del\Lm_n$}
\end{array}
\right\}.
\end{align*}
For $w>0$, we further define
\[
\cT_w\!:=\!\{\text{$\exists$ an infinite open $k$-tube of width at most $w$ whose initial sphere intersects $\Lm_o$}\}.
\]
Then clearly, $\cT=\bigcup_{w>0}\cT_w$.
To prove Proposition~\ref{prop:approx}, it suffices to show that, for each $w>0$,
\begin{equation}\label{eq:approx_Tw}
\cT_w
\subset\bigcap_{n\in\N}\cT_{w,n}
\subset\cT.
\end{equation}
Indeed, taking the union of~\eqref{eq:approx_Tw} over all $w>0$ yields~\eqref{eq:approx}.
The first inclusion in~\eqref{eq:approx_Tw} is immediate; therefore, the remainder of this section is devoted to proving the second inclusion in~\eqref{eq:approx_Tw}.

\subsection{Associated Directed Graph on Collars}
As noted in Remark~\ref{rem:transitivity}, in contrast to ordinary paths in bond percolation, the class of $k$-tubes is not closed under concatenation, making it nontrivial to construct an infinite open $k$-tube from a sequence of open $k$-tubes.
To address this difficulty, we introduce a directed graph associated with each configuration of the $k$-plaquette percolation model.
Intuitively speaking, in the case $k=2$, we view a topological loop together with a local trajectory passing through it~(a \emph{collar}) as a vertex.
We then connect two nearby collars by a directed edge if and only if there exists a tube from one to the other such that, near its initial and terminal loops, the tube is compatible with the local trajectories of the collars.
This additional structure enables us to concatenate tubes without violating the local self-avoidance condition at the junction.

We first introduce the notion of \emph{collars}.
In the following, we call a topological $(k-1)$-sphere contained in the $(k-1)$-dimensional complete cubical set $|\cK_{k-1}^d|$ a \emph{discrete $(k-1)$-sphere}.
Equivalently, a discrete $(k-1)$-sphere is a $(k-1)$-dimensional cubical set in $\R^d$ that is homeomorphic to $\S^{k-1}$.
\begin{df}[$k$-collar]
  Let $I_\eps:=[-\eps,\eps]$ for a fixed $\eps>0$.
  Given a discrete $(k-1)$-sphere $S$, a self-avoiding $k$-tube
  \[
  \dl\colon\S^{k-1}\times I_\eps\to|\cK_k^d|
  \]
  such that $\dl(\S^{k-1}\times\{0\})=S$ is called a \emph{$k$-collar} based at $S$.
  We refer to the discrete $(k-1)$-sphere $S$ as the \emph{base sphere} of $\dl$~(and as the \emph{base loop} when $k=2$), and denote it by $S_\dl$.
\end{df}
\begin{df}[Equivalence class of $k$-collars]
\label{df:equiv_collar}
  We say that two $k$-collars
  \[
  \dl\colon\S^{k-1}\times I_\eps\to|\cK_k^d|
  \quad\text{and}\quad
  \dl'\colon\S^{k-1}\times I_\eps\to|\cK_k^d|
  \]
  based at the same discrete $(k-1)$-sphere $S$ are \emph{equivalent} if there exists a continuous map
  \[
  H\colon\bigl(\S^{k-1}\times I_\eps\bigr)\times[0,1]\to|\cK_k^d|
  \]
  such that
  \begin{itemize}
    \item $H(\cdot,0)=\dl$ and $H(\cdot,1)=\dl'$, that is, $H$ is a homotopy between $\dl$ and $\dl'$;
    \item the map $H(\cdot,t)\colon\S^{k-1}\times I_\eps\to|\cK_k^d|$ is a $k$-collar based at $S$ for every $t\in[0,1]$~(in particular, each $H(\cdot,t)$ is injective).
  \end{itemize}

  We denote by $[\dl]$ the equivalence class of a $k$-collar $\dl$.
  For an equivalence class $\Dl=[\dl]$, we set $S_\Dl:=S_\dl$ and refer to it as the \emph{base sphere} of $\Dl$.
  (See Figure~\ref{fig:collars}.)
\end{df}
Intuitively, the equivalence class of a $k$-collar encodes only the directions from which the base sphere come and to which it proceeds.
\begin{figure}[H]
  \centering
  \includegraphics[width=5.5cm]{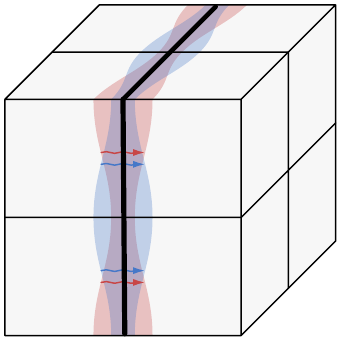}
  \caption{Illustration of two equivalent collars~(whose images are shown in red and blue, with red and blue arrows indicating the direction of the trajectories) for $k=2$ and $d=3$.
  The bold black line represents their common base loop~(a $2\times2$ square loop).
  They are equivalent only when the red and blue trajectories pass through the base loop in the same direction.}
  \label{fig:collars}
\end{figure}

We now define a directed graph associated with a give $k$-plaquette configuration.
\begin{df}[Associated directed graph on collars]
\label{df:associated}
  Let $w>0$.
  Associated with a configuration $\eta\in\Om$ of the $k$-plaquette percolation model on $\R^d$, a directed graph $G_w=G_w(\eta)$ is defined as follows.
  The vertex set $V_w$ is given by
  \[
  V_w:=\{[\dl]\mid\text{$\dl$ is a $k$-collar based at some discrete $(k-1)$-sphere $S$ with $\diam(S)\le w$}\}.
  \]
  For distinct vertices $\Dl,\Dl'\in V_w$, we include the ordered pair $(\Dl,\Dl')$ as a directed edge in $E_w=E_w(\eta)$ if and only if there exist $k$-collars
  \[
  \dl\colon\S^{k-1}\times I_\eps\to|\cK_k^d|
  \quad\text{and}\quad
  \dl'\colon\S^{k-1}\times I_\eps\to|\cK_k^d|
  \]
  representing the equivalence classes $\Dl$ and $\Dl'$, respectively, and an open $k$-tube
  \[
  \gm\colon\S^{k-1}\times[a,b]\to\R^d
  \quad\text{with $\width(\gm)\le w$}
  \]
  such that the following conditions hold:
  \begin{itemize}
    \item (\emph{Initial compatibility}).
    For any $(p,h)\in\S^{k-1}\times[0,\eps]$,
    \[
    \gm(p,a+h)=\dl(p,h);
    \]
    \item (\emph{Terminal compatibility}).
    For any $(p,h)\in\S^{k-1}\times[-\eps,0]$,
    \[
    \gm(p,b+h)=\dl'(p,h);
    \]
    \item (\emph{Locality}).
    It holds that
    \[
    \diam(\im\gm)\le2w+4.
    \]
  \end{itemize}
\end{df}
Note that the directed graph $G_w=(V_w,E_w)$ is locally finite; that is, for each $\Dl\in V_w$, the numbers of in-edges and out-edges are finite.
This follows immediately from the locality condition in Definition~\ref{df:associated} and the fact that there are only finitely many equivalence classes of $k$-collars based at the same discrete $(k-1)$-sphere, which ultimately is a consequence of the discrete structure of the $k$-dimensional complete cubical set $|\cK_k^d|$.

\subsection{Reachability and Join}
We consider two types of relations: a relation on the associated directed graph $G_w=(V_w,E_w)$, defined via the existence of directed paths, and a relation on topological spheres, defined via the existence of connecting open tubes.
\begin{df}[Reachability]\label{df:reachability}
  For two vertices $\Dl$ and $\Dl'$ in the associated directed graph $G_w$, we say that $\Dl'$ is \emph{reachable} from $\Dl$ in $G_w$, and write $\Dl\rightsquigarrow_w\Dl'$, if there exists a directed path from $\Dl$ to $\Dl'$; that is, a finite sequence $(\Dl_0,\Dl_1,\ldots,\Dl_I)$ of vertices in $G_w$ such that $\Dl_0=\Dl$, $\Dl_I=\Dl'$, and $(\Dl_{i-1},\Dl_i)\in E_w$ for all $i=1,2,\ldots,I$.
  Furthermore, we say $\Dl$ \emph{reaches infinity} in $G_w$, and write $\Dl\rightsquigarrow_w\infty$, if there exists an infinite self-avoiding directed path from $\Dl$; that is, a sequence $(\Dl_0,\Dl_1,\ldots)$ of distinct vertices such that $\Dl_0=\Dl$ and $(\Dl_{i-1},\Dl_i)\in E_w$ for all $i\in\N$.
\end{df}

\begin{df}[Join]
  Two topological $(k-1)$-spheres $S$ and $S'$ in $\R^d$ are said to be \emph{joined with width at most $w$}, denoted by $S\approx_wS'$, if there exists an open $k$-tube of width at most $w$ whose initial sphere is $S$ and terminal sphere is $S'$.
  We say that $S$ is \emph{joined to infinity with width at most $w$}, and write $S\approx_w\infty$, if there exists an infinite open $k$-tube of width at most $w$ whose initial sphere is $S$.
  Furthermore, if there exits some $w>0$ such that $S\approx_w\infty$, then we simply say that $S$ is \emph{joined to infinity}, and denote this by $S\approx\infty$.
\end{df}

The following lemma clarifies the relationship between the reachability relation $\rightsquigarrow_w$ in the associated directed graph $G_w$ and the join relation $\approx_w$ in the $k$-plaquette percolation model.
Recall that $S_\Dl$ denotes the base sphere of $\Dl\in V_w$~(see Definition~\ref{df:equiv_collar}).
\begin{lem}\label{lem:reach-join}
  For each $w>0$, the following statements hold$:$
  \begin{itemize}
    \item $($From reachability to join$)$.
    If $\Dl\rightsquigarrow_w\Dl'$, then
    \[
    S_\Dl\approx_wS_{\Dl'}.
    \]
    \item $($From join to reachability$)$.
    Let $S,S'$ be discrete $(k-1)$-spheres.
    If $S\approx_wS'$, then there exist $\Dl,\Dl'\in V_{2w+4}$ with $S_\Dl=S$ and $S_{\Dl'}=S'$ such that
    \[
    \Dl\rightsquigarrow_{2w+4}\Dl'.
    \]
  \end{itemize}
\end{lem}
\begin{rem}
  Whereas the width parameter $w$ is preserved when passing from reachability to join, it may increase when passing from join to reachability.
  This possible increase arises from the discretization of $k$-tubes required in the proof.
  Indeed, although the initial and terminal spheres of an open $k$-tube $\gm\colon\S^{k-1}\times[a,b]\to\R^d$ are discrete $(k-1)$-spheres, an intermediate topological $(k-1)$-sphere $\gm(\S^{k-1}\times\{t\})$ at level $t$ need not be a discrete $(k-1)$-sphere in general.
  Accordingly, modifying such a $k$-tube to form discrete $(k-1)$-spheres at intermediate levels may incur an additional width cost; however, this increase remains controlled and is sufficient for our purposes.
\end{rem}
\begin{proof}[Proof of Lemma~\ref{lem:reach-join}]
  (\emph{From reachability to join}).
  Suppose that $\Dl\rightsquigarrow_w\Dl'$.
  Then there exist a directed path $(\Dl_0,\Dl_1,\ldots,\Dl_I)$ in $G_w$ with $\Dl_0=\Dl$ and $\Dl_I=\Dl'$.
  By the definition of the directed-edge set $E_w$, for each $i=1,2,\ldots,I$, we can take $k$-collars
  \[
  \dl_{i-1}\colon\S^{k-1}\times I_\eps\to|\cK_k^d|
  \quad\text{and}\quad
  \dl'_i\colon\S^{k-1}\times I_\eps\to|\cK_k^d|
  \]
  representing the equivalence classes $\Dl_{i-1}$ and $\Dl_i$, respectively, together with an open $k$-tube
  \[
  \gm_i\colon\S^{k-1}\times[i-1,i]\to\R^d
  \quad\text{with $\width(\gm_i)\le w$}
  \]
  such that the initial and terminal compatibility conditions in Definition~\ref{df:associated} are satisfied:
  \begin{align*}
  \gm_i(p,i-1+h)&=\dl_{i-1}(p,h)\text{ for any $(p,h)\in\S^{k-1}\times[0,\eps]$}
  \shortintertext{and}
  \gm_i(p,i+h)&=\dl'_i(p,h)\text{ for any $(p,h)\in\S^{k-1}\times[-\eps,0]$.}
  \end{align*}
  Fix $i=1,2,\ldots,I-1$.
  Since $[\dl'_i]=\Dl_i=[\dl_i]$, there exists a homotopy
  \[
  H\colon\bigl(\S^{k-1}\times I_\eps\bigr)\times[0,1]\to|\cK_k^d|
  \]
  between $\dl_{i-1}$ and $\dl'_i$ satisfying the additional conditions in Definition~\ref{df:equiv_collar}.
  In particular, for sufficiently small $h\in(0,\eps]$, we have
  \[
  \dl'_i(\S^{k-1}\times[-h,0])\,\cap\,\dl_i(\S^{k-1}\times[0,h])=S_{\Dl_i}.
  \]
  Therefore, we can adjust only the initial part $\gm_{i+1}|_{\S^{k-1}\times[i,i+h]}$ of $\gm_{i+1}$ so that the terminal part $\gm_i|_{\S^{k-1}\times[i-h,i]}$ of $\gm_i$ attaches continuously to $\gm_{i+1}$, while preserving local self-avoidance.
  This adjustment can be performed without changing the width.
  Denote the resulting maps by $\widetilde\gm_i$~($i=2,3,\ldots,I$), and set $\widetilde\gm_1:=\gm_1$.
  We then define a new open $k$-tube $\gm\colon\S^{k-1}\times[0,I]\to\R^d$ by
  \[
  \gm(p,t):=\widetilde\gm_i(p,t)\quad\text{for $t\in[i-1,i]$.}
  \]
  The initial and terminal spheres of the open $k$-tube $\gm$ is $S_\Dl$ and $S_{\Dl'}$, respectively.
  This completes the proof of the first statement.

  \medskip
  \noindent(\emph{From join to reachability}).
  Let $S,S'$ be discrete $(k-1)$-spheres, and let
  \[
  \gm\colon\S^{k-1}\times[a,b]\to\R^d
  \]
  be an open $k$-tube of width at most $w$ with initial sphere $S$ and terminal sphere $S'$.
  We inductively define a finite sequence $(S_0,S_1,S_2,\ldots)$ of topological $(k-1)$-spheres so that they are sufficiently well separated along the tube, which will be convenient for the discretization later.
  Set $t_0:=a$ and $S_0:=S$.
  For each $i\ge1$, define
  \begin{equation}\label{eq:t_i}
  t_i:=\inf\{t\in[t_{i-1},b]\mid d_\infty(\gm(\S^{k-1}\times\{t\}),S_{i-1})\ge2\}
  \quad\text{and}\quad
  S_i:=\gm(\S^{k-1}\times\{t_i\}).
  \end{equation}
  Here, $d_\infty(A,B):=\inf\{\|x-y\|_\infty\mid x\in A,\,y\in B\}$ is the $\ell^\infty$-distance between subsets $A$ and $B$ of $\R^d$.
  If the set in~\eqref{eq:t_i} is empty, then we define $t_i:=b$ and $S_i=S'$, and the procedure terminates.
  Suppose that the above procedure yields sequences
  \[
  a=t_0<t_1<\cdots<t_I=b
  \quad\text{and}\quad
  S=S_0,S_1,\ldots,S_I=S'
  \]
  for some $I\in\N$.
  Note that, except for $S_0=S$ and $S_I=S'$, the intermediate topological $(k-1)$-spheres $S_i$~($i=1,2,\ldots,I-1$) are not necessarily discrete $(k-1)$-spheres.
  
  We now modify each restricted $k$-tube
  \[
  \gm_i:=\gm|_{\S^{k-1}\times[t_{i-1},t_i]}
  \]
  so that its initial and terminal spheres $S_{i-1}$ and $S_i$ become discrete $(k-1)$-spheres.
  To this end, let $X_i$ denote the smallest cubical set containing $S_i$~(see Figure~\ref{fig:discretize}).
  Note that $X_i$ is a $k$-dimensional cubical set (or $(k-1)$-dimensional if $S_i$ is already a discrete $(k-1)$-sphere).
  For $i=0,1,\ldots,I$ and $p\in\S^{k-1}$, define
  \[
  \tau_i(p):=
  \begin{cases}
    t_i     &\text{if $i\in\{0,I\}$,}\\
    \sup\{t\in[0,t_i]\mid\gm(p,t)\notin X_i\}   &\text{if $1\le i\le I-1$.}
  \end{cases}
  \]
  In other words, $\tau_i(p)$ is the first time $t$ at which $\gm(p,t)$ exits $X_i$ as $t$ decreases from $t_i$.
  In general, $\tau_i(p)$ might not be continuous in $p$, since for a fixed $p\in\S^{k-1}$, the trajectory $t\mapsto\gm(p,t)\in|\cK_k^d|$ may fail to intersect the boundary of $X_i$ transversely.
  However, by perturbing the original map $\gm$ at an arbitrarily small additional width cost $\a>0$, one can ensure that all such trajectories intersect $|\cK_{k-1}^d|$, and hence the boundary of $X_i$, transversely; it then follows that $\tau_i(\cdot)$ is continuous.
  Hence, we assume that $\gm$ satisfies this condition, with $\width(\gm)\le w':=w+\a$.
  Since $X_i$ is a union of plaquettes (each of diameter $1$) intersecting $S_i$, the construction of $S_i$ implies that $t_{i-1}<\tau_i(p)\le t_i$ for each $i=1,2,\ldots,I$ and any $p\in\S^{k-1}$.
  Consequently, for every $p\in\S^{k-1}$,
  \[
  a=\tau_0(p)=t_0<\tau_1(p)\le t_1<\tau_2(p)\le t_2<\cdots\le t_{I-1}<\tau_I(p)=t_I=b.
  \]
  For each $i=1,2,\ldots,I$, we define a modified open $k$-tube
  \[
  \widetilde\gm_i\colon\S^{k-1}\times[t_{i-1},t_i]\to\R^d
  \]
  associated with $\gm_i$ by
  \[
  \widetilde\gm_i(p,t):=\gm\biggl(p,\tau_{i-1}(p)+\frac{\tau_i(p)-\tau_{i-1}(p)}{t_i-t_{i-1}}(t-t_{i-1})\biggr).
  \]
  That is, for each $p\in\S^{k-1}$, we linearly rescale the interval $[t_{i-1},t_i]$ onto $[\tau_{i-1}(p),\tau_i(p)]$, and then apply the original map $\gm$.
  With this construction, the map $\widetilde\gm_i$ is continuous, and its initial and terminal spheres, denoted by $\widetilde S_{i-1}$ and $\widetilde S_i$, respectively, are both discrete $(k-1)$-spheres.
  We then define a $k$-collar $\dl_i\colon\S^{k-1}\times I_\eps\to|\cK_k^d|$ based at $\widetilde S_i$ along $\widetilde\gm_i$~($i=1,2,\ldots,I$).
  
  Finally, we show that the finite sequence $([\dl_0],[\dl_1],\dots,[\dl_I])$ forms a directed path in $G_{2w+4}$.
  Since
  \begin{align*}
  \im\widetilde\gm_i
  &=\gm\Biggl(\bigsqcup_{p\in\S^{k-1}}\bigl(\{p\}\times[\tau_{i-1}(p),\tau_i(p)]\bigr)\Biggr)\\
  &\subset\gm\Biggl(
  \bigsqcup_{p\in\S^{k-1}}\bigl(\{p\}\times[\tau_{i-1}(p),t_{i-1}]\bigr)\cup\bigl(\S^{k-1}\times[t_{i-1},t_i]\bigr)\Biggr)\\
  &=\gm\Biggl(\bigsqcup_{p\in\S^{k-1}}\bigl(\{p\}\times[\tau_{i-1}(p),t_{i-1}]\bigr)\Biggr)\cup\im\gm_i\\
  &\subset X_{i-1}\cup\im\gm_i,
  \end{align*}
  it follows that for any $x,y\in\im\widetilde\gm_i$,
  \[
  \|x-y\|_\infty\le
  \begin{cases}
    w'+2         &\text{if $x,y\in X_{i-1}$,}\\
    2w'+2        &\text{if $x,y\in\im\gm_i$,}\\
    2w'+3        &\text{otherwise.}
  \end{cases}
  \]
  The first case follows from the definition of $t_i$ in~\eqref{eq:t_i}, while the second follows from the definition of $X_i$.
  For the remaining case, suppose that $x\in X_{i-1}$ and $y\in\im\gm_i$.
  Then choose a point $z\in S_{i-1}\subset X_{i-1}\cap\im\gm_i$ such that $\|x-z\|_\infty\le1$, and apply the triangle inequality.
  Therefore,
  \[
  \diam(\im\widetilde\gm_i)
  \le2w'+3.
  \]
  Thus, the locality condition in Definition~\ref{df:associated} is satisfied by taking $\a=1/2$.
  Moreover,
  \[
  \width(\widetilde\gm_i)
  \le\diam(\im\widetilde\gm_i)
  \le2w'+3
  =2w+4,
  \]
  which completes the proof.
\end{proof}

\begin{figure}[H]
  \centering
  \includegraphics[width=4.5cm]{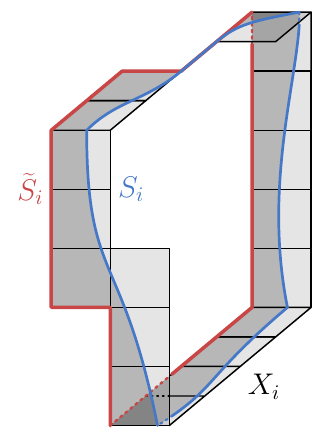}
  \caption{Illustration of the discretization of the intermediate topological $(k-1)$-sphere $S_i$ for $k=2$ and $d=3$.
  The loop $S_i$ (shown in blue) is suitably perturbed to a discrete loop $\widetilde S_i$ (shown in red) in the cubical set $X_i$.}
  \label{fig:discretize}
\end{figure}

By the same arguments as in the proof of Lemma~\ref{lem:reach-join}, we obtain the following analogue.
\begin{lem}\label{lem:r-j_infty}
  For each $w>0$, the following statements hold$:$
  \begin{itemize}
    \item $($From reachability to join$)$.
    If $\Dl\rightsquigarrow_w\infty$, then
    \[
    S_\Dl\approx_w\infty.
    \]
    \item $($From join to reachability$)$.
    Let $S$ be a discrete $(k-1)$-sphere.
    If $S\approx_w\infty$, then there exists $\Dl\in V_{2w+4}$ with $S_\Dl=S$ such that
    \[
    \Dl\rightsquigarrow_{2w+4}\infty.
    \]
  \end{itemize}
\end{lem}
\begin{proof}
  (\emph{From reachability to join}).
  As in the proof of Lemma~\ref{lem:reach-join}, we construct an infinite open $k$-tube $\gm\colon\S^{k-1}\times[0,\infty)\to\R^d$ using the modified maps $\widetilde\gm_1,\widetilde\gm_2,\ldots$.
  The only point that requires additional verification is the escape-to-infinity condition:
  \[
  \lim_{t\to\infty}d_\infty(o,\gm(\S^{k-1}\times\{t\}))=\infty.
  \]
  This follows immediately from the fact that the base sphere $S_{\Dl_i}$ escapes to infinity as $i\to\infty$, together with the locality condition satisfied by each $\gm_i$.
  
  \smallskip
  \noindent(\emph{From join to reachability}).
  The argument is entirely analogous to that of Lemma~\ref{lem:reach-join}.
\end{proof}

\subsection{Proof of Proposition~\ref{prop:approx}}
We introduce the following additional notation.
For a vertex $\Dl$ in the associated directed graph $G_w=(V_w,E_w)$, let $\cC_w^+(\Dl)$ denote the \emph{out-component} of $\Dl$; that is,
\[
\cC_w^+(\Dl)
:=\{\Dl'\in V_w\mid\Dl\rightsquigarrow_w\Dl'\}.
\]
We are now in a position to prove Proposition~\ref{prop:approx}.
\begin{proof}[Proof of Proposition~\ref{prop:approx}]
  As noted at the beginning of Section~\ref{sec:approx}, it suffices to establish the second inclusion in~\eqref{eq:approx_Tw} for a fixed $w>0$.
  In particular, we show that
  \begin{equation}\label{eq:approx_Tw_2nd}
  \bigcap_{n\in\N}\cT_{w,n}\subset\cT_{4w+13}.
  \end{equation}
  
  Suppose that the left-hand event occurs, and fix a sufficiently large $n\in\N$.
  Then there exists an open $k$-tube $\gm_n$ of width at most $w$ whose initial sphere intersects $\Lm_o$ and terminal sphere intersects $\del\Lm_n$.
  Since the initial and terminal spheres of $\gm_n$ are not necessarily discrete $(k-1)$-spheres, we slightly modify $\gm_n$ so that they become discrete $(k-1)$-spheres in order to apply Lemma~\ref{lem:reach-join}~(from join to reachability).
  The discretization argument is similar to that in the lemma, namely, by shrinking each of the initial and terminal spheres within its smallest cubical covering.
  The additional width cost is bounded in the same way; namely, replace $w$ with $2w+4$.
  We then obtain discrete $(k-1)$-spheres $S_n$ and $S'_n$ such that
  \[
  S_n\cap\Lm_1\neq\emptyset,\quad
  S'_n\cap(\Lm_{n-1}^\circ)^c\neq\emptyset,\quad
  \text{and}\quad S_n\approx_{2w+4}S'_n.
  \]
  Here, $\Lm_{n-1}^\circ$ denotes the interior of $\Lm_{n-1}$.
  By Lemma~\ref{lem:reach-join}~(from join to reachability), there exist $\Dl_n,\Dl'_n\in V_{2(2w+4)+4}=V_{4w+12}$ with $S_{\Dl_n}=S_n$ and $S_{\Dl'_n}=S'_n$ such that
  \[
  \Dl_n\rightsquigarrow_{4w+12}\Dl'_n.
  \]
  Since there are only finitely many (depending on $w$) vertices in $V_{4w+12}$ whose base spheres intersect $\Lm_1$, and $n$ can be taken arbitrarily large, it follows that $|\cC_{4w+12}^+(\Dl_N)|=\infty$ for some $N\in\N$.
  Since the directed graph $G_{4w+12}$ is locally finite, this implies~(e.g.,~by K\H onig's lemma) that
  \[
  \Dl_N\rightsquigarrow_{4w+12}\infty.
  \]
  Applying Lemma~\ref{lem:r-j_infty}~(from reachability to join), we conclude that
  \[
  S_{\Dl_N}\approx_{4w+12}\infty.
  \]
  Recalling that $S_{\Dl_N}=S_N$, we obtain
  \[
  S_N\approx_{4w+12}\infty
  \quad\text{with}\quad
  S_N\cap\Lm_1\neq\emptyset.
  \]
  Finally, if necessary, we may modify the infinite open $k$-tube witnessing the above join relation so that its initial sphere intersects $\Lm_o$, at the cost of increasing the width by at most one.
  Indeed, since the discrete $(k-1)$-sphere $S_N$ was obtained by shrinking the original $k$-tube $\gm_N$, there exists an incident open $k$-plaquette intersecting $\Lm_o$, allowing a slight modification of the initial sphere.
  Therefore, the right-hand event in~\eqref{eq:approx_Tw_2nd} occurs, completing the proof.
\end{proof}

\section{Sharp Threshold for Tubular One-Arm Events}
\label{sec:sharp_1-arm}
In this section, we prove Theorem~\ref{thm:sharp_1-arm}.
Recall that, for each $w>0$ and $n\in\N$,
\[
\theta_{w,n}^{\tube}(p)
\!:=\P_p(\cT_{w,n})
=\P_p\left(
\begin{array}{l}
  \text{$\exists$ an open $k$-tube of width at most $w$ whose initial sphere}\\
  \text{intersects $\Lm_o$ and terminal sphere intersects $\del\Lm_n$}
\end{array}
\right).
\]
As noted below~\eqref{eq:1-arm_event}, when $n\le1/2+w$, we set $\theta_{w,n}^{\tube}(p)=1$ by convention.
We further define
\[
\theta_w^{\tube}(p)
:=\lim_{n\to\infty}\theta_{w,n}^{\tube}(p)
\]
and the \emph{critical probability} $p_c^{\tube}(w)=p_c^{\tube}(w,k,d)$ by
\[
p_c^{\tube}(w):=\inf\{p\in[0,1]\mid\theta_w^{\tube}(p)>0\}.
\]
Note that $\lim_{w\to\infty}\theta_w^{\tube}(p)=\theta^{\tube}(p)$ by Proposition~\ref{prop:approx}.
Therefore, $p_c^{\tube}(w)\searrow p_c^{\tube}$ as $w\to\infty$.

The following proposition shows that the tubular one-arm event exhibits a sharp threshold at $p_c^{\tube}(w)$: below $p_c^{\tube}(w)$, the tubular one-arm probability $\theta_{w,n}^{\tube}(p)$ decays exponentially in $n$, whereas above $p_c^{\tube}(w)$, it admits a mean-field-type lower bound.
\begin{prop}\label{prop:sharp_1-arm_w}
  Let $1\le k\le d-1$ and $w>0$.
  Then the following statements hold$:$
  \begin{enumerate}
    \item If $p<p_c^{\tube}(w)$, then there exists a constant $c_{p,w}>0$ such that
    \[
    \theta_{w,n}^{\tube}(p)\le\exp(-c_{p,w}n)
    \]
    for any integer $n>1/2+w$.
    \item There exists a constant $\kappa>0$, independent of $w$, such that
    \[
    \theta_w^{\tube}(p)\ge\kappa(p-p_c^{\tube}(w))
    \]
    for all $p>p_c^{\tube}(w)$.
  \end{enumerate}
\end{prop}

Theorem~\ref{thm:sharp_1-arm} follows immediately from Proposition~\ref{prop:sharp_1-arm_w}.
\begin{proof}[Proof of Theorem~\ref{thm:sharp_1-arm}]
  (1)~Suppose that $p<p_c^{\tube}$, and fix $w>0$.
  Since $p<p_c^{\tube}\le p_c^{\tube}(w)$, applying Proposition~\ref{prop:sharp_1-arm_w}(1) yields the first statement.

  (2)~Let $\kappa>0$ be the constant appearing in Proposition~\ref{prop:sharp_1-arm_w}(2), and suppose that $p>p_c^{\tube}$.
  Since $p_c^{\tube}(w)\searrow p_c^{\tube}$ as $w\to\infty$, we can choose $w$ sufficiently large so that $p_c^{\tube}\le p_c^{\tube}(w)<p$.
  Applying Proposition~\ref{prop:sharp_1-arm_w}(2), we have
  \[
  \theta_w^{\tube}(p)\ge\kappa(p-p_c^{\tube}(w)).
  \]
  Letting $w\to\infty$ on both sides of the above inequality yields
  \[
  \theta^{\tube}(p)\ge\kappa(p-p_c^{\tube}),
  \]
  which completes the proof.
\end{proof}

To prove Proposition~\ref{prop:sharp_1-arm_w}, we apply the OSSS inequality, whose effectiveness depends on the choice of algorithm that sequentially reveals the bits of restricted $k$-plaquette configuration $\eta\in\{0,1\}^{\cK_k^d(\Lm_n)}$ until the occurrence of the tubular one-arm event $\cT_{w,n}$ is determined.
In the next section, we briefly review the OSSS inequality.

\subsection{The OSSS Inequality}
Let $N\in\N$.
An \emph{algorithm} $\cA=(\psi_t)_{t=1}^N$ receives $\eta\in\{0,1\}^N$ as input and produces a sequence $(i_1,i_2,\ldots,i_N)$ of distinct indices in $\{1,2,\ldots,N\}$, according to the following inductive procedure:
for each $t=1,2,\ldots,N$,
\[
i_t=\psi_t(i_1,i_2,\ldots,i_{t-1};\eta_{i_1},\eta_{i_2},\ldots,\eta_{i_{t-1}}),
\]
where each $\psi_t$ is a deterministic rule that selects the next index $i_t$ to be queried based on the previously queried indices $i_1,i_2,\ldots,i_{t-1}$ and their revealed values $\eta_{i_1},\eta_{i_2},\ldots,\eta_{i_{t-1}}$.
Given a Boolean function $f\colon\{0,1\}^N\to\{0,1\}$ and an input $\eta\in\{0,1\}^N$, we define a \emph{stopping time} $\tau_f=\tau_f(\eta)$ as the smallest $0\le t\le N$ such that for all $x\in\{0,1\}^N$ satisfying $(x_{i_1},x_{i_2},\ldots,x_{i_t})=(\eta_{i_1},\eta_{i_2},\ldots,\eta_{i_t})$, it holds that $f(x)=f(\eta)$.
For an algorithm $\cA$, the probability
\[
\dl_p(i,f;\cA):=\P_p(i_t=i\text{ for some }t\le\tau_f)
\]
is called the \emph{revealment probability} of the $i$th bit for $f$ under the algorithm $\cA$.

\begin{lem}[The OSSS inequality~\cite{OSSS2005}]\label{lem:OSSS}
  Let $f\colon\{0,1\}^N\to\{0,1\}$ be an increasing Boolean function.
  For any algorithm $\cA$,
  \[
  \Var_p(f)\le p(1-p)\sum_{i=1}^N\dl_p(i,f;\cA)\Inf_p(i,f).
  \]
  Here,
  \[
  \Inf_p(i,f):=\E_p[|f-(f\circ\mathrm{Flip}_i)|]
  \]
  denotes the influence of the $i$th bit on $f$, where $\mathrm{Flip}_i\colon\{0,1\}^N\to\{0,1\}^N$ is the map that flips the $i$th value.
\end{lem}

For the proof of Theorem~\ref{thm:sharp_1-arm}, the following differentiation formula is also useful.
\begin{lem}[The Margulis--Russo formula~\cite{Ma1974,Ru1981}]\label{lem:Russo}
  Let $f\colon\{0,1\}^N\to\{0,1\}$ be an increasing Boolean function.
  Then,
  \[
  \frac d{dp}\E_p[f]=\sum_{i=1}^N\Inf_p(i,f)
  \]
\end{lem}
\begin{rem}
  Once we construct an algorithm $\cA$ that determines $f_N$ and whose revealment probabilities $\dl_p(i,f_N;\cA)$ ($i=1,2,\ldots,N$) are all small (say, $p(1-p)\dl_p(i,f_N;\cA)\le1/N$), Lemmas~\ref{lem:OSSS} and~\ref{lem:Russo} together imply that
  \[
  \frac d{dp}\E_p[f_N]
  \ge N\Var_p(f_N)
  =N\E_p[f_N](1-\E_p[f_N]).
  \]
  Therefore,
  \[
  \frac d{dp}\biggl(\log\frac{\E_p[f_N]}{1-\E_p[f_N]}\biggr)
  =\frac{\frac d{dp}\E_p[f_N]}{\E_p[f_N](1-\E_p[f_N])}
  \ge N.
  \]
  Assume that there exists $p_c\in(0,1)$ and constants $0<\a\le\b<1$ such that $\a\le\E_{p_c}[f_N]\le\b$ for all $N\in\N$.
  Then for any fixed $\dl>0$, integrating the above differential inequality over $[p_c-\dl,p_c]$ and $[p_c,p_c+\dl]$ yields, respectively,
  \[
  \E_{p_c-\dl}[f_N]\le\frac\b{1-\b}\exp(-\dl N)
  \quad\text{ and }\quad
  \E_{p_c+\dl}[f_N]\ge1-\frac{1-\a}\a\exp(-\dl N).
  \]
  This implies that the Boolean function $f_N$ exhibits a sharp threshold at $p_c$ as $N\to\infty$.
  However, identifying such $p_c$ is generally difficult unless the model possesses an appropriate self-duality.
  We therefore follow an approach based on a general calculus lemma~(Lemma~\ref{lem:calculus}), developed in~\cite{DRT2019}.
\end{rem}

\subsection{Algorithm for Proposition~\ref{prop:sharp_1-arm_w}}
In this section, we describe a family $\{\cA_m\}_{m=1}^n$ of algorithms that determine a Boolean function $f_{w,n}:=\1_{\cT_{w,n}}$, where $\cT_{w,n}$ is the tubular one-arm event.
We may regard the Boolean function $f_{w,n}$ as a function on $\{0,1\}^{\cK_k^d(\Lm_n)}$.
We write $N:=|\cK_k^d(\Lm_n)|$.

Now, for a fixed $m=1,2,\ldots,n$, we define an algorithm $\cA_m$ that produces a tuple
\[
(Q_1,Q_2,\ldots,Q_N)
\]
of distinct $k$-plaquettes in $\Lm_n$, while simultaneously constructing a growing sequence:
\[
\cS_0\subset\cS_1\subset\cdots\subset\cS_N\subset\scrS_{k-1},
\]
where $\scrS_{k-1}$ denotes the set of all topological $(k-1)$-spheres in $\R^d$.
We set
\[
\cS_0=\{S\in\scrS_{k-1}\mid S\cap\del\Lm_m\neq\emptyset\,\text{ and }\diam(S)\le w\}.
\]
For each step $t=1,2,\ldots,N$, we define $\cS_t$ and $Q_t$ from the already constructed sequences
\[
\cS_0\subset\cS_1\subset\cdots\subset\cS_{t-1}\quad\text{and}\quad(Q_1,Q_2,\ldots,Q_{t-1})
\]
as follow:
\begin{itemize}
  \item[(Case~1)] If there exists a $k$-plaquette $Q\in\cK_k^d(\Lm_n)\setminus\{Q_1,Q_2,\ldots,Q_{t-1}\}$ that intersects a topological sphere $S\in\cS_{t-1}$, then choose the smallest such $k$-plaquette $Q$ with respect to a prescribed total order on $\cK_k^d(\Lm_n)$ and set $Q_t:=Q$, and define
  \[
  \cS_t:=\cS_{t-1}\cup\left\{S\in\scrS_{k-1}\relmiddle|
  \begin{array}{l}
    \text{$S$ is joined to a topological sphere in $\cS_0$ by an open}\\
    \text{$k$-tube of width at most $w$ that is supported on the}\\
    \text{already queried $k$-plaquettes $\{Q_1,Q_2,\ldots,Q_{t-1},Q\}$}
  \end{array}
  \right\};
  \]
  \item[(Case~2)] If no such $k$-plaquette as in Case~1 exists, then set $Q_t$ to be the smallest $Q\in\cK_k^d(\Lm_n)\setminus\{Q_1,Q_2,\ldots,Q_{t-1}\}$ with respect to the prescribed total order on $\cK_k^d(\Lm_n)$, and set $\cS_t:=\cS_{t-1}$.
\end{itemize}
Note that once Case~1 is completed, Case~2 applies to all subsequent steps.
Furthermore, the stopping time $\tau_{f_{w,n}}$ for this algorithm $\cA_m$ is no later than the last time Case~1 occurs.
If a plaquette $Q\in\cK_k^d(\Lm_n)$ is queried until the stopping time $\tau_{f_{w,n}}$ (i.e., $Q_t=Q$ for some $t\le\tau_{f_{w,n}}$), then by the choice of $Q_t$ in Case~1, there exists an $S\in\scrS_{k-1}$ intersecting $Q$ such that $S$ is joined to a topological sphere in $\cS_0$ by an open $k$-tube of width at most $w$.
This observation provides the following upper bound on the revealment probability of $Q\in\cK_k^d(\Lm_n)$ for $f_{w,n}$ under the algorithm $\cA_m$:
\begin{align*}
&\dl_p(Q,f_{w,n};\cA_m)\\
&\le\P_p\Biggl(\bigcup_{x\in Q\cap\Z^d}\left\{
\begin{array}{l}
  \text{$\exists$ an open $k$-tube of width at most $w$ whose initial sphere}\\
  \text{intersects $x+\Lm_o$ and terminal sphere intersects $\del\Lm_m$}
\end{array}
\right\}\Biggr)\\
&\le\sum_{x\in Q\cap\Z^d}\theta_{w,|m-\|x\|_\infty|}^{\tube}(p).
\end{align*}
Therefore, for each $Q\in\cK_k^d(\Lm_n)$,
\begin{align*}
\sum_{m=1}^n\dl_p(Q,f_{w,n};\cA_m)
&\le\sum_{m=1}^n\sum_{x\in Q\cap\Z^d}\theta_{w,|m-\|x\|_\infty|}^{\tube}(p)\\
&\le2\sum_{x\in Q\cap\Z^d}\sum_{j=0}^{n-1}\theta_{w,j}^{\tube}(p)\\
&\le2^{k+1}\sum_{j=0}^{n-1}\theta_{w,j}^{\tube}(p).
\end{align*}
Combining the above estimate with Lemmas~\ref{lem:OSSS} and~\ref{lem:Russo}, we obtain
\begin{align*}
\Var_p(f_{w,n})
&\le p(1-p)\sum_{Q\in\cK_k^d(\Lm_n)}\Biggl(\frac1n\sum_{m=1}^n\dl_p(Q,f_{w,n};\cA_m)\Biggr)\Inf_p(Q,f_{w,n})\\
&\le2^{k+1}\Biggl(\frac1n\sum_{j=0}^{n-1}\theta_{w,j}^{\tube}(p)\Biggr)\cdot\frac d{dp}\theta_{w,n}^{\tube}(p).
\end{align*}
Let $0<p_{\max}<1$ be fixed.
Then for any $p\in[0,p_{\max}]$ and integer $n>1/2+w$, we have
\[
\Var_p(f_{w,n})
=\theta_{w,n}^{\tube}(p)(1-\theta_{w,n}^{\tube}(p))
\ge\theta_{w,n}^{\tube}(p)(1-p)^{\binom{2d}k}
\ge\theta_{w,n}^{\tube}(p)(1-p_{\max})^{\binom{2d}k}.
\]
For the first inequality above, we used the fact that $f_{w,n}:=\1_{\cT_{w,n}}=0$ whenever there exists no open $k$-plaquette containing the origin.
Consequently, we obtain the following self-referential differential inequality on the interval $(0,p_{\max})$ for each integer $n>1/2+w$:
\begin{equation}\label{eq:diff_ineq}
\frac d{dp}\theta_{w,n}^{\tube}(p)
\ge\frac{(1-p_{\max})^{\binom{2d}k}}{2^{k+1}}\Biggl(\frac1n\sum_{j=0}^{n-1}\theta_{w,j}^{\tube}(p)\Biggr)^{-1}\theta_{w,n}^{\tube}(p).
\end{equation}

\subsection{Proof of Proposition~\ref{prop:sharp_1-arm_w}}
We use the following fundamental lemma to solve the self-referential differential inequality~\eqref{eq:diff_ineq}.
We state it in a form convenient for our purpose.
\begin{lem}[{\cite[Lemma~3.1]{DRT2019}}]\label{lem:calculus}
  Let $(g_n)_{n=0}^\infty$ be a pointwise convergent sequence of differentiable, nondecreasing functions from $[0,p_{\max}]$ to $[0,M]$.
  Let $n_0\in\N$, and suppose that for $n\ge n_0$, the following differential inequality holds$:$
  \[
  \frac d{dp}g_n(p)\ge\Biggl(\frac1n\sum_{j=0}^{n-1}g_j(p)\Biggr)^{-1}g_n(p)\quad\text{for any $p\in(0,p_{\max})$}.
  \]
  Then there exists $\widetilde p\in[0,p_{\max}]$ such that the following statements hold$:$
  \begin{itemize}
    \item[(i)] If $p<\widetilde p$, then there exists a constant $c_{p,n_0}>0$ such that
    \[
    g_n(p)\le M\exp(-c_{p,n_0}n)
    \]
    for all $n\ge n_0$$;$
    \item[(ii)] For any $p>\widetilde p$, it holds that $\lim_{n\to\infty}g_n(p)\ge p-\widetilde p$.
  \end{itemize}
\end{lem}

We are now ready to proof Proposition~\ref{prop:sharp_1-arm_w}.
\begin{proof}[Proof of Proposition~\ref{prop:sharp_1-arm_w}]
  When $w<1$, the statement is trivial.
Indeed, $\theta_{w,n}^{\tube}(p)=0$ for all $p\in[0,1]$ and all integers $n>1/2+w$, and consequently $p_c^{\tube}(w)=1$.
  We therefore assume that $w\ge1$.
  As seen in the proof of~\eqref{eq:s-tube},
  \[
  p_c^{\tube}(w)\le p_c^{\tube}(1)<1.
  \]
  Hence, we may choose $p_{\max}$ independently of $w$ such that $p_c^{\tube}(1)<p_{\max}<1$.
  Applying Lemma~\ref{lem:calculus} to~\eqref{eq:diff_ineq} with
  \[
  g_n(p)=\frac{2^{k+1}}{(1-p_{\max})^{\binom{2d}k}}\theta_{w,n}^{\tube}(p),
  \quad M=\frac{2^{k+1}}{(1-p_{\max})^{\binom{2d}k}},
  \quad\text{and}\quad n_0=\lfloor1/2+w\rfloor+1,
  \]
  we obtain $\widetilde p\in[0,p_{\max}]$ such that the following hold:
  \begin{itemize}
    \item If $p<\widetilde p$, then there exists a constant $c_{p,w}>0$ such that
    \[
    \theta_{w,n}^{\tube}(p)\le\exp(-c_{p,w}n)
    \]
    for all integers $n>1/2+w$;
    \item For any $p>\widetilde p$,
    \[
    \theta_w^{\tube}(p)\ge\kappa(p-\widetilde p),
    \]
    where $\kappa=(1-p_{\max})^{\binom{2d}k}/2^{k+1}$.
  \end{itemize}
  In particular, since $\theta_w^{\tube}(p)=\lim_{n\to\infty}\theta_{w,n}^{\tube}(p)$ is either zero for $p<\tilde p$ or positive for $p>\tilde p$, it follows by definition that $p_c^{\tube}(w)=\widetilde p$.
  Moreover, the constant $\kappa$ is independent of $w$, which completes the proof.
\end{proof}

\section{An Analogue of the Uniqueness of the Infinite Cluster}
\label{sec:uniqueness}
In this section, we prove Theorem~\ref{thm:uniqueness}.
To this end, we first establish a corresponding statement for the associated directed graph $G_w=G_w(\eta)$, which will also be used in Section~\ref{sec:box-crossing}.

\subsection{Reduction to the Associated Directed Graph}
In what follows, an infinite self-avoiding directed path $(\Dl_0,\Dl_1,\Dl_2,\ldots)\in V_w^{\Z_{\ge0}}$ as in Definition~\ref{df:reachability} will be called an \emph{out-ray} in $G_w$.
Similarly, an \emph{in-ray} in $G_w$ is a backward infinite self-avoiding directed path $(\ldots,\Dl_{-2},\Dl_{-1},\Dl_0)\in V_w^{\Z_{\le0}}$, and a \emph{double ray} in $G_w$ is a bi-infinite self-avoiding directed path $(\ldots,\Dl_{-1},\Dl_0,\Dl_1,\ldots)\in V_w^\Z$.
\begin{lem}\label{lem:uniqueness}
  Let $w>0$, and let $G_w=(V_w,E_w)$ be the associated directed graph.
  Then for every $p\in[0,1]$, the following holds $\P_p$-almost surely$:$
  If
  \[
  \pi^-=(\ldots,\Dl_{-2}^-,\Dl_{-1}^-,\Dl_0^-)\in V_w^{\Z_{\le0}}
  \quad\text{and}\quad
  \pi^+=(\Dl_0^+,\Dl_1^+,\Dl_2^+,\ldots)\in V_w^{\Z_{\ge0}}
  \]
  are in-ray and out-ray, respectively, then there exists a double ray
  \[
  \pi=(\ldots,\Dl_{-1},\Dl_0,\Dl_1,\ldots)\in V_w^\Z
  \]
  such that
  \begin{equation}\label{eq:tail_ray}
  \pi|_{\Z\cap(-\infty,-I]}=\pi^-|_{\Z\cap(-\infty,-I]}
  \quad\text{and}\quad
  \pi|_{\Z\cap[I,\infty)}=\pi^+|_{\Z\cap[I,\infty)}
  \end{equation}
  for some $I\in\N_0$.
\end{lem}

To prove Lemma~\ref{lem:uniqueness}, we employ a Burton--Keane-type argument based on the notion of trifurcation~\cite{BK1989}.
However, substantial modifications are required, since $G_w$ is a directed graph and one cannot consider its connected components directly.
To this end, we begin by reviewing the notion of \emph{ends for directed graphs}, as introduced in~\cite{Zu1998}.
For two out-rays $\pi$ and $\pi'$ in $G_w$, we write $\pi\lesssim\pi'$ if there exist infinitely many vertex-disjoint paths whose initial vertices lie in $\pi$ and terminal vertices lie in $\pi'$.
The relation $\lesssim$ is a preorder on the set of all out-rays.
We then define an equivalence relation $\sim$ by declaring that $\pi\sim\pi'$ if and only if $\pi\lesssim\pi'$ and $\pi'\lesssim\pi$.
Equivalently, $\pi\sim\pi'$ if and only if there exists an out-ray $\pi''$ in $G_w$ that contains infinitely many vertices of both of $\pi$ and $\pi'$.
For each out-ray $\pi$ in $G_w$, we call its equivalence class $[\pi]$ an \emph{out-end} of $G_w$.
Let $\Ends^+(G_w)$ denote the set of all out-ends of $G_w$.
The preorder $\lesssim$ on out-rays induces a partial order $\le$ on $\Ends^+(G_w)$.
The above notions are defined similarly for in-rays in $G_w$, and we obtain a partially ordered set $(\Ends^-(G_w),\le)$ as well.
In~\cite{Zu1998}, in-rays and out-rays are not distinguished; for our purposes, however, it is more convenient to treat them separately.

With this notion in place, we now prove Lemma~\ref{lem:uniqueness}.
We do so by showing that, almost surely, there are no pairs of in-ends and out-ends in $G_w$ that are not connected by a double ray.
\begin{proof}[Proof of Lemma~\ref{lem:uniqueness}]
  The statement is trivial if $p=0,1$, and hence we assume $0<p<1$.
  We say that a pair $([\pi^-],[\pi^+])\in\Ends^-(G_w)\times\Ends^+(G_w)$ is \emph{undesired} if there exists no double ray $\pi\in V_w^\Z$ satisfying~\eqref{eq:tail_ray} for some $I\in\N_0$.
  This notion is well defined, since the validity of~\eqref{eq:tail_ray} for some $I\in\N_0$ is independent of the choice of representatives of $[\pi^-]$ and $[\pi^+]$.
  Let $M$ denote the supremum over all integers $m\ge0$ for which there exist undesired pairs
  \[
  ([\pi_1^-],[\pi_1^+]),([\pi_2^-],[\pi_2^+]),\ldots,([\pi_m^-],[\pi_m^+])
  \]
  such that $[\pi_1^\pm],[\pi_2^\pm],\ldots,[\pi_m^\pm]$ are pairwise incomparable in $(\Ends^\pm(G_w),\le)$.
  Since $M$ is translation invariant (that is, $M=M\circ\tau_x$ for every translation $\tau_x\colon\Om\to\Om$ by $x\in\Z^d$), ergodicity implies that $M$ is almost surely constant.
  Therefore, there exists $m\in\{0\}\cup\N\cup\{\infty\}$ such that $\P_p(M=m)=1$.
  To prove the lemma, it suffices to rule out the cases $m\in\N$ and $m=\infty$, thereby concluding that $m=0$.

  \medskip
  We first rule out the case $m\in\N$.
  Assume that $m\in\N$.
  Let $E_n$ denote the event that, for every undesired pair $([\pi^-],[\pi^+])\in\Ends^-(G_w)\times\Ends^+(G_w)$, there exist representatives $\pi^-$ and $\pi^+$ of the corresponding ends such that the base spheres of the terminal vertex of the in-ray $\pi^-$ and the initial vertex of the out-ray $\pi^+$ both intersect $\Lm_n$.
  Since $m$ is finite,
  \[
  1=\P_p(M=m)\le\P_p\Biggl(\bigcup_{n\in\N}E_n\Biggr)=\lim_{n\to\infty}\P_p(E_n).
  \]
  Therefore, we may choose $n\in\N$ sufficiently large so that $\P_p(E_n)>0$.
  Define a modified configuration $\widetilde\eta$ by
  \[
  \widetilde\eta(Q):=
  \begin{cases}
    1         &\text{if $Q\in\cK_k^d(\Lm_{n+\lceil2w+4\rceil})$,}\\
    \eta(Q)    &\text{otherwise.}
  \end{cases}
  \]
  Then
  \begin{align*}
  \P_p(M=0)
  &\ge\P_p\bigl(\widetilde\eta\in\{M=0\},\,\eta(Q)=1\text{ for all }Q\in\cK_k^d(\Lm_{n+\lceil2w+4\rceil})\bigr)\\
  &=\P_p(\widetilde\eta\in\{M=0\})\cdot p^{|\cK_k^d(\Lm_{n+\lceil2w+4\rceil})|}\\
  &\ge\P_p(E_n)\cdot p^{|\cK_k^d(\Lm_{n+\lceil2w+4\rceil})|}
  >0,
  \end{align*}
  which contradicts $\P_p(M=m)=1$.

  \medskip
  It remains to rule out the possibility that $m=\infty$.
  Assume that $m=\infty$.
  In what follows, let $G_w^\text{und}=G_w^\text{und}(\eta)$ denote the undirected graph obtained from the associated directed graph $G_w=G_w(\eta)$ by forgetting the orientation of all directed edges (the vertex set remains $V_w$).
  For a $(k+1)$-plaquette $R$, we denote by $Q_1,Q_2,\ldots,Q_{2(k+1)}$ the $k$-plaquettes that form the faces of $R$.
  Note that the boundary of each $Q_i$ is a discrete $(k-1)$-sphere, which we denote by $\del Q_i$.
  Let $\cC(R)\subset V_w$ denote the set of all vertices that are connected in $G_w^\text{und}$ to a vertex whose base sphere is $\del Q_i$ for some $i=1,2,\ldots,2(k+1)$.
  We call a $(k+1)$-plaquette $R$ a \emph{trifurcation} if the following conditions hold~(see Figure~\ref{fig:trifurcation} for an example):
  \begin{itemize}
  \item[(i)] The $k$-plaquettes $Q_1,Q_2,\ldots,Q_{2(k+1)}$ are all open, and $\cC(R)$ contains exactly one infinite connected component of the undirected graph $G_w^\text{und}$;
  \item[(ii)] Let $\eta_R$ be the configuration obtained from $\eta$ by declaring $Q_1,Q_2,\ldots,Q_{2(k+1)}$ closed.
  Then, after identifying all vertices whose base spheres are equal to $\del Q_i$ for each $i=1,2,\ldots,2(k+1)$ in the undirected graph $G_w^\text{und}(\eta_R)$, the infinite connected component of $\cC(R)$ splits into exactly three distinct infinite connected components, along with finitely many additional components, all of whose vertices have base spheres contained in the $\lceil2w+4\rceil$-thickening of $R$.
  \end{itemize}
  
  We now fix a specific $(k+1)$-plaquette $R_o:=[0,1]^{k+1}\times\{0\}^{d-k-1}$, and prove that
  \begin{equation}\label{eq:pos_density}
  \P_p(\text{$R_o$ is a trifurcation})>0.
  \end{equation}
  Since $\P_p(M=\infty)=1$, we can choose $n\in\N$ sufficiently large so that
  \begin{equation}\label{eq:3_out-rays}
  \P_p\left(
  \begin{array}{l}
    \text{$\exists$ out-rays $\pi_1,\pi_2,\pi_3$ whose corresponding out-ends are pairwise}\\
    \text{incomparable in $(\Ends^+(G_w),\le)$, and such that the base sphere}\\
    \text{of the initial vertex of each $\pi_h$~($h=1,2,3$) intersects $\Lm_n$}
  \end{array}
  \right)
  >0.
  \end{equation}
  Since $[\pi_1],[\pi_2],[\pi_3]$ in the above event are pairwise incomparable in $(\Ends^+(G_w),\le)$, there exist only finitely many vertex-disjoint paths between $\pi_h$ and $\pi_{h'}$~($h\neq h'$).
  Hence, by enlarging $n$ if necessary, we may assume that the base spheres of the vertices appearing in these finitely many vertex-disjoint paths (for every pair among $\pi_1,\pi_2,\pi_3$) intersect $\Lm_n$.
  By truncating each $\pi_h$ at the last vertex whose base sphere intersects $\Lm_n$, we may further assume that the following conditions hold for the out-rays $\pi_h=(\Dl_{h,0},\Dl_{h,1},\Dl_{h,2},\ldots)$~($h=1,2,3$) in the event~\eqref{eq:3_out-rays}:
  \begin{itemize}
    \item[(C1)] For each $h=1,2,3$, only the initial vertex has a base sphere intersecting $\Lm_n$; that is,
    \[
    S_{\Dl_{h,0}}\cap\Lm_n\neq\emptyset
    \quad\text{and}\quad
    S_{\Dl_{h,i}}\cap\Lm_n=\emptyset\text{ for $i\ge1$.}
    \]
    \item[(C2)] For each $h=1,2,3$, write $\pi_h\setminus\Dl_{h,0}:=(\Dl_{h,1},\Dl_{h,2},\ldots)$.
    Then the out-rays $\pi_1\setminus\Dl_{1,0}$, $\pi_2\setminus\Dl_{2,0}$, and $\pi_3\setminus\Dl_{3,0}$ lie in three distinct connected components of the subgraph of $G_w^\text{und}$ induced by $\{\Dl\in V_w\mid S_\Dl\cap\Lm_n=\emptyset\}$.
  \end{itemize}
  We then define the event
  \[
  F_n:=\left\{
  \begin{array}{l}
    \text{there exist out-rays $\pi_h=(\Dl_{h,0},\Dl_{h,1},\Dl_{h,2},\ldots)$~($h=1,2,3$) in $G_w$}\\
    \text{such that Conditions~(C1) and~(C2) are satisfied}
  \end{array}
  \right\}.
  \]
  By~\eqref{eq:3_out-rays} and the preceding discussion, we obtain $\P_p(F_n)>0$.
  We now deduce~\eqref{eq:pos_density} by a local modification argument.
  For each $\eta\in F_n$, we construct a modified configuration $T(\eta)$ that coincides with $\eta$ outside $\Lm_{n+\lceil2w+4\rceil}$ and for which $R_o$ is a trifurcation~(see Figure~\ref{fig:trifurcation}).\footnote{Whenever necessary, we may choose the initial vertices of $\pi_h$~($h=1,2,3$) so that the pairwise $\ell^\infty$-distances between their base spheres are sufficiently large.
  This can be achieved by a minor adjustment of the preceding argument, since $\P_p(M=\infty)=1$.}
  Outside $F_n$, define $T$ arbitrarily, requiring only that $T(\eta)$ coincides with $\eta$ outside $\Lm_{n+\lceil2w+4\rceil}$.
  Then $F_n\subset T^{-1}(A_o)$, where $A_o:=\{\text{$R_o$ is a trifurcation}\}$, and hence $\P_p(T^{-1}(A_o))\ge\P_p(F_n)>0$.
  Since $\P_p$ is a product measure with $p\in(0,1)$ and $T$ modifies only finitely many $k$-plaquettes within $\Lm_{n+\lceil2w+4\rceil}$, such modifications preserve positivity; more precisely,
  \[
  \P_p(A_o)
  \ge\min\{p,1-p\}^{|\cK_k^d(\Lm_{n+\lceil2w+4\rceil})|}\,\P_p(T^{-1}(A_o))>0.
  \]

  We now complete the proof by deriving a contradiction.
  Let $T_n$ denote the number of trifurcation $(k+1)$-plaquettes contained in $\Lm_n$.
  By translation invariance, we have
  \[
  \E_p[T_n]
  =\sum_{R\in\cK_{k+1}^d(\Lm_n)}\P_p(\text{$R$ is a trifurcation})
  =|\cK_{k+1}^d(\Lm_n)|\cdot\P_p(\text{$R_o$ is a trifurcation}).
  \]
  Since $\P_p(\text{$R_o$ is a trifurcation})>0$ by~\eqref{eq:pos_density}, the right-hand side grows on the order of $n^d$ as $n\to\infty$.
  On the other hand, we shall use the following deterministic estimate, proved immediately after the present proof:
  $T_n$ grows at most on the order of $n^{d-1}$ as $n\to\infty$.
  This contradicts the above estimate on $\E_p[T_n]$ and thereby completes the proof.
\end{proof}

\begin{figure}[H]
  \centering
  \includegraphics[width=14cm]{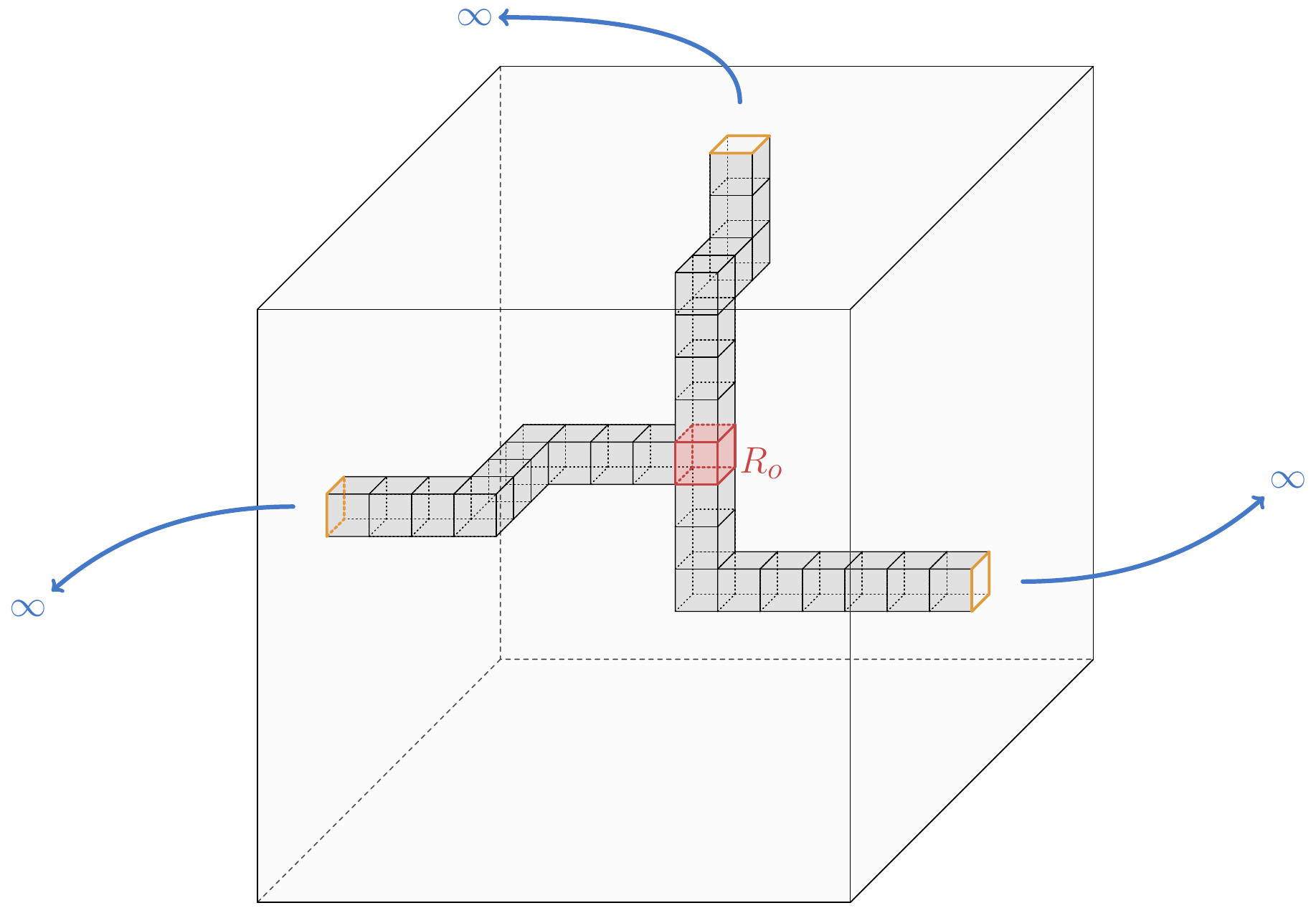}
  \caption{Illustration of the trifurcation $(k+1)$-plaquette $R_o$ (shown in red)  for $k=2$ and $d=3$, obtained after modifying the $k$-plaquette configuration inside $\Lm_{n+\lceil2w+4\rceil}$.
  The blue arrows indicate the out-rays $\pi_h\setminus\Dl_{h,0}$~($h=1,2,3$), as constructed in the proof of Lemma~\ref{lem:uniqueness}.
  Vertices of the undirected graph $G_w^\text{und}$ whose base loops are shown in orange reach infinity along these out-rays.}
  \label{fig:trifurcation}
\end{figure}

We conclude this section by proving the following deterministic estimate, used in the proof of Lemma~\ref{lem:uniqueness}.
\begin{lem}
  Let $T_n$ denote the number of trifurcation $(k+1)$-plaquettes contained in $\Lm_n$.
  Then there exists a constant $C>0$ such that $T_n\le Cn^{d-1}$ for all $n\in\N$.
\end{lem}
\begin{proof}
  Fix an integer $n>\lceil2w+4\rceil$.
  Let $\cC$ be a connected component of the subgraph of the undirected graph $G_w^\text{und}$ induced by $\{\Dl\in V_w\mid S_\Dl\subset\Lm_{n+\lceil2w+4\rceil}\}$.
  Suppose that a $(k+1)$-plaquette $R\subset\Lm_{n-\lceil2w+4\rceil}$ is a trifurcation with $\cC(R)\cap\cC\neq\emptyset$.
  Then, declaring its $2(k+1)$ faces to be closed, together with the vertex identification in the definition of trifurcation, induces a $3$-partition
  \[
  \Pi(R)=\{\Sg_1(R),\Sg_2(R),\Sg_3(R)\}
  \]
  of a set $\Sg_\cC:=\cC\cap\{\Dl\in V_w\mid S_\Dl\subset\Lm_{n+\lceil2w+4\rceil}\setminus\Lm_n\}$, by the definition of trifurcation.
  Indeed, each of the three infinite connected components appearing in the definition intersects $\Sg_\cC$, whereas the remaining ``small'' components do not meet $\Sg_\cC$, since $R\subset\Lm_{n-\lceil2w+4\rceil}$.
  Moreover, the family
  \[
  \{\Pi(R)\colon\text{$R\in\cK_{k+1}^d(\Lm_{n-\lceil2w+4\rceil})$ is a trifurcation with $\cC(R)\cap\cC\neq\emptyset$}\}
  \]
  of $3$-partitions of the set $\Sg_\cC$ is \emph{compatible}, in the sense that for any two distinct $3$-partitions
  \[
  \Pi(R)=\{\Sg_1(R),\Sg_2(R),\Sg_3(R)\}
  \quad\text{and}\quad
  \Pi(R')=\{\Sg_1(R'),\Sg_2(R'),\Sg_3(R')\},
  \]
  there exists a relabeling of the elements such that $\Sg_1(R)\supset\Sg_2(R')\,\cup\,\Sg_3(R')$.
  It therefore follows from the classical bound on the cardinality of a compatible family of $3$-partitions~(see, e.g.,~\cite[Lemma~8.5]{Gr1999}) that
  \[
  |\{R\in\cK_{k+1}^d(\Lm_{n-\lceil2w+4\rceil})\mid\text{$R$ is a trifurcation with $\cC(R)\cap\cC\neq\emptyset$}\}|\le|\Sg_\cC|.
  \]
  Summing this inequality over all connected components $\cC$ of the subgraphs of $G_w^\text{und}$ induced by $\{\Dl\in V_w\mid S_\Dl\subset\Lm_{n+\lceil2w+4\rceil}\}$ yields
  \[
  |\{R\in\cK_{k+1}^d(\Lm_{n-\lceil2w+4\rceil})\mid\text{$R$ is a trifurcation}\}|\le|\{\Dl\in V_w\mid S_\Dl\subset\Lm_{n+\lceil2w+4\rceil}\setminus\Lm_n\}|.
  \]
  Finally, recall that there are only finitely many equivalence classes of $k$-collars based at a given discrete $(k-1)$-sphere, independently of $n$.
  It follows that the right-hand side is bounded above by a constant (independent of $n$) times $n^{d-1}$ for all $n\in\N$.
\end{proof}

\subsection{Proof of Theorem~\ref{thm:uniqueness}}
Here we prove Theorem~\ref{thm:uniqueness} using Lemma~\ref{lem:uniqueness}.
\begin{proof}[Proof of Theorem~\ref{thm:uniqueness}]
  Let
  \[
  \gm^-\colon\S^{k-1}\times(-\infty,0]\to\R^d
  \quad\text{and}\quad
  \gm^+\colon\S^{k-1}\times[0,\infty)\to\R^d
  \]
  be backward and forward infinite open $k$-tubes of width at most $w$, respectively.
  By the same argument as in the proof of Lemma~\ref{lem:reach-join}~(from join to reachability), we can discretize these tubes.
  More specifically, we construct sequences
  \[
  0<t_1^-<t_2^-<\cdots
  \quad\text{and}\quad
  0<t_1^+<t_2^+<\cdots,
  \]
  together with homeomorphisms (parameter perturbations) $\Phi^-$ and $\Phi^+$ defined on $\S^{k-1}\times(-\infty,0]$ and $\S^{k-1}\times[0,\infty)$, respectively, such that:
  \begin{itemize}
    \item for every $p\in\S^{k-1}$, the map $\Phi^-(p,\cdot)$ is nondecreasing with $\Phi^-(p,0)=(p,0)$, and $\Phi^+(p,\cdot)$ is nonincreasing with $\Phi^+(p,0)=(p,0)$;
    \item for each $i\in\N$, the perturbed $(k-1)$-sphere
    \[
    \widetilde S_i^-:=\gm^-(\Phi^-(\S^{k-1}\times\{t_i^-\}))
    \quad\text{and}\quad
    \widetilde S_i^+:=\gm^+(\Phi^+(\S^{k-1}\times\{t_i^+\}))
    \]
    are discrete $(k-1)$-spheres approximating the original topological $(k-1)$-spheres
    \[
    S_i^-:=\gm^-(\S^{k-1}\times\{t_i^-\})
    \quad\text{and}\quad
    S_i^+:=\gm^+(\S^{k-1}\times\{t_i^+\}),
    \]
    respectively, in the sense that $d_H(S_i^-,\widetilde S_i^-)\le1$ and $d_H(S_i^+,\widetilde S_i^+)\le1$, where $d_H$ denotes the Hausdorff distance;
    \item there exist backward and forward infinite directed paths
    \[
    \pi^-=(\ldots,\Dl_2^-,\Dl_1^-)\in V_{2w+4}^{\Z_{\le0}}
    \quad\text{and}\quad
    \pi^+=(\Dl_1^+,\Dl_2^+,\ldots)\in V_{2w+4}^{\Z_{\ge0}},
    \]
    respectively, in $G_{2w+4}$ such that $S_{\Dl_i^-}=\widetilde S_i^-$ and $S_{\Dl_i^+}=\widetilde S_i^+$ for all $i\in\N$.
  \end{itemize}
  On the almost sure event in Lemma~\ref{lem:uniqueness}, applied with $2w+4$ in place of $w$, we obtain a bi-infinite directed path
  \[
  \pi=(\ldots,\Dl_{-1},\Dl_0,\Dl_1,\ldots)\in V_{2w+4}^\Z
  \]
  such that
  \[
  \pi|_{\Z\cap(-\infty,-I]}=\pi^-|_{\Z\cap(-\infty,-I]}
  \quad\text{and}\quad
  \pi|_{\Z\cap[I,\infty)}=\pi^+|_{\Z\cap[I,\infty)}
  \]
  for some $I\in\N$.
  If necessary, we first modify $\pi^-$ and $\pi^+$ to be self-avoiding, noting that each vertex appears only finitely many times in these paths.
  Using the same concatenation argument as in Lemma~\ref{lem:reach-join}~(from reachability to join), we can construct a bi-infinite open $k$-tube $\widetilde\gm\colon\S^{k-1}\times\R\to\R^d$ of width at most $2w+4$ such that the tails of the perturbed tubes $\gm^-\circ\Phi^-$ and $\gm^+\circ\Phi^+$ are combined by $\widetilde\gm$:
  \[
  \widetilde\gm|_{\S^{k-1}\times(-\infty,-t_I^-]}=(\gm^-\circ\Phi^-)|_{\S^{k-1}\times(-\infty,-t_I^-]}
  \quad\text{and}\quad
  \widetilde\gm|_{\S^{k-1}\times[t_I^+,\infty)}=(\gm^+\circ\Phi^+)|_{\S^{k-1}\times[t_I^+,\infty)}.
  \]
  Finally, we slightly adjust $\widetilde\gm$ using the homeomorphisms $\Phi^-$ and $\Phi^+$ so that the perturbed spheres $\widetilde S_I^-$ and $\widetilde S_I^+$ coincide with the original topological $(k-1)$-spheres $S_I^-$ and $S_I^+$, respectively, at the cost of an additional width depending only on $w$.
  This yields the desired bi-infinite open $k$-tube.
\end{proof}

\section{Sharp Threshold for the Tubular Box-Crossing Property}
\label{sec:box-crossing}
In this section, we prove Theorem~\ref{thm:box-crossing}.
Recall that the \emph{tubular box-crossing event} $\cH_{w,n}$ is defined by
\[
\cH_{w,n}:=\left\{
  \begin{array}{l}
    \text{$\exists$ an open $k$-tube in $\Lm_n$ of width at most $w$ whose initial sphere}\\
    \text{intersects $L_n$ and terminal sphere intersects $R_n$}
  \end{array}
  \right\},
\]
where $L_n:=\{-n\}\times[-n,n]^{d-1}$ and $R_n:=\{n\}\times[-n,n]^{d-1}$ denote the left and right faces of the box $\Lm_n$, respectively.
The first statement of Theorem~\ref{thm:box-crossing} is an immediate consequence from the exponential decay of the tubular one-arm probability in $n$ when $p<p_c^{\tube}$~(Theorem~\ref{thm:sharp_1-arm}(1)).
\begin{proof}[Proof of Theorem~\ref{thm:box-crossing}(1)]
  Suppose that $p<p_c^{\tube}$, and let $w>0$.
  Since
  \[
  \cH_{w,n}\subset\bigcup_{x\in L_n\cap\Z^d}\left\{
  \begin{array}{l}
    \text{$\exists$ an open $k$-tube of width at most $w$ whose initial sphere}\\
    \text{intersects $x+\Lm_o$ and terminal sphere intersects $\del(x+\Lm_{2n})$}
  \end{array}
  \right\},
  \]
  the union bound, together with translation invariance, implies
  \begin{equation}\label{eq:box-crossing1}
  \P_p(\cH_{w,n})
  \le|L_n\cap\Z^d|\cdot\theta_{w,2n}^{\tube}(p).
  \end{equation}
  By Theorem~\ref{thm:sharp_1-arm}(1), there exists a constant $c_{p,w}>0$ such that
  \[
  \theta_{w,2n}^{\tube}(p)\le\exp(-c_{p,w}n)
  \]
  for all $n\in\N$ with $2n>1/2+w$.
  Since $|L_n\cap\Z^d|$ grows only polynomially in $n$, it follows that the right-hand side of~\eqref{eq:box-crossing1} tends to zero as $n\to\infty$.
  This completes the proof.
\end{proof}

To prove the second statement of Theorem~\ref{thm:box-crossing}, we exploit the analogue of the uniqueness of the infinite open cluster established in the previous section~(Lemma~\ref{lem:uniqueness}).
\begin{proof}[Proof of Theorem~\ref{thm:box-crossing}(2)]
  Suppose that $p>p_c^{\tube}$, and fix an arbitrary constant $h\in(0,1]$.
  Since there exists an infinite open $k$-tube almost surely by~\eqref{eq:tube_perc}, we can choose $m\in\N$ such that
  \[
  \P_p(I_m):=\P_p(\text{$\exists$ an infinite open $k$-tube whose initial sphere intersects $\Lm_m$})
  >1-h.
  \]
  For $w>0$ and an integer $n>m+w$, let $F_{n,1},F_{n,2},\ldots,F_{n,2d}$ denote the $2d$ faces of the box $\Lm_n$.
  Set $F_{n,1}:=L_n$ and $F_{n,2}:=R_n$.
  For each $i=1,2,\ldots,2d$, we define the event
  \[
  \cF_{w,n,i}:=\left\{
  \begin{array}{l}
    \text{$\exists$ an open $k$-tube in $\Lm_n$ of width at most $w$ whose initial}\\
    \text{sphere intersects $\Lm_m$ and terminal sphere intersects $F_{n,i}$}
  \end{array}
  \right\}.
  \]
  Since
  \[
  1-h<\P_p(I_m)
  \le\P_p\Biggl(\bigcup_{w\in\N}\,\,\bigcap_{n>m+w}\,\,\bigcup_{i=1}^{2d}\cF_{w,n,i}\Biggr)
  =\lim_{w\to\infty}\P_p\Biggl(\bigcap_{n>m+w}\,\,\bigcup_{i=1}^{2d}\cF_{w,n,i}\Biggr),
  \]
  we can take large $w>0$ such that for all integers $n>m+w$,
  \[
  \P_p\Biggl(\bigcup_{i=1}^{2d}\cF_{w,n,i}\Biggr)\ge1-h.
  \]
  Therefore, for every integer $n>m+w$,
  \[
  h\ge\P_p\Biggl(\bigcap_{i=1}^{2d}\cF_{w,n,i}^{\,c}\Biggr)
  \ge\prod_{i=1}^{2d}\P_p(\cF_{w,n,i}^c)
  =\{1-\P_p(\cF_{w,n,1})\}^{2d}.
  \]
  Here, the second inequality follows from the FKG inequality, since each event $\cF_{w,n,i}$ is increasing events.
  For the last equality uses the rotational symmetry of the faces $F_{n,i}$.
  Thus,
  \[
  \P_p(\cF_{w,n,1})=\P_p(\cF_{w,n,2})\ge1-h^{1/(2d)}.
  \]
  By the FKG inequality again, we obtain
  \begin{equation}\label{eq:box-crossing2}
  \P_p(\cF_{w,n,1}\cap\cF_{w,n,2})
  \ge\P_p(\cF_{w,n,1})\P_p(\cF_{w,n,2})
  \ge\bigl(1-h^{1/(2d)}\bigr)^2
  \end{equation}
  for all integers $n>m+w$.
  
  We estimate the probability that two open $k$-tubes witnessing the event $\cF_{w,n,1}\cap\cF_{w,n,2}$ cannot be combined to form a single open $k$-tube from $F_{n,1}=L_n$ to $F_{n,2}=R_n$~(see Figure~\ref{fig:crossing}).
  Suppose that the event $\cF_{w,n,1}\cap\cF_{w,n,2}\setminus\cH_{4w+13,n}$ occurs for infinitely many integers $n>m+w$.
  Then, as in the proof of~\eqref{eq:approx_Tw_2nd}, we can construct in-ray $\pi^-$ and out-ray $\pi^+$ in $G_{4w+13}$ such that there exists no double ray $\pi$ in $G_{4w+13}$ satisfying~\eqref{eq:tail_ray} for some $I\in\N_0$.
  This event has probability zero by Lemma~\ref{lem:uniqueness}.
  Therefore,
  \[
  \limsup_{n\to\infty}\P_p(\cF_{w,n,1}\cap\cF_{w,n,2}\setminus\cH_{4w+13,n})
  \le\P_p\Bigl(\limsup_{n\to\infty}(\cF_{w,n,1}\cap\cF_{w,n,2}\setminus\cH_{4w+13,n})\Bigr)
  =0
  \]
  Combining this with~\eqref{eq:box-crossing2}, we obtain
  \begin{align*}
  \P_p(\cH_{4w+13,n})
  &\ge\P_p(\cH_{4w+13,n}\cap(\cF_{w,n,1}\cap\cF_{w,n,2}))\\
  &=\P_p(\cF_{w,n,1}\cap\cF_{w,n,2})-\P_p(\cF_{w,n,1}\cap\cF_{w,n,2}\setminus\cH_{4w+13,n})\\
  &\ge\bigl(1-h^{1/(2d)}\bigr)^2-\P_p(\cF_{w,n,1}\cap\cF_{w,n,2}\setminus\cH_{4w+13,n})\\
  &\xrightarrow[n\to\infty]{}\bigl(1-h^{1/(2d)}\bigr)^2.
  \end{align*}
  Consequently,
  \[
  \liminf_{n\to\infty}\P_p(\cH_{4w+13,n})
  \ge\bigl(1-h^{1/(2d)}\bigr)^2.
  \]
  Since $h>0$ is arbitrary, the conclusion follows.
\end{proof}

\begin{figure}[H]
  \centering
  \includegraphics[width=9.5cm]{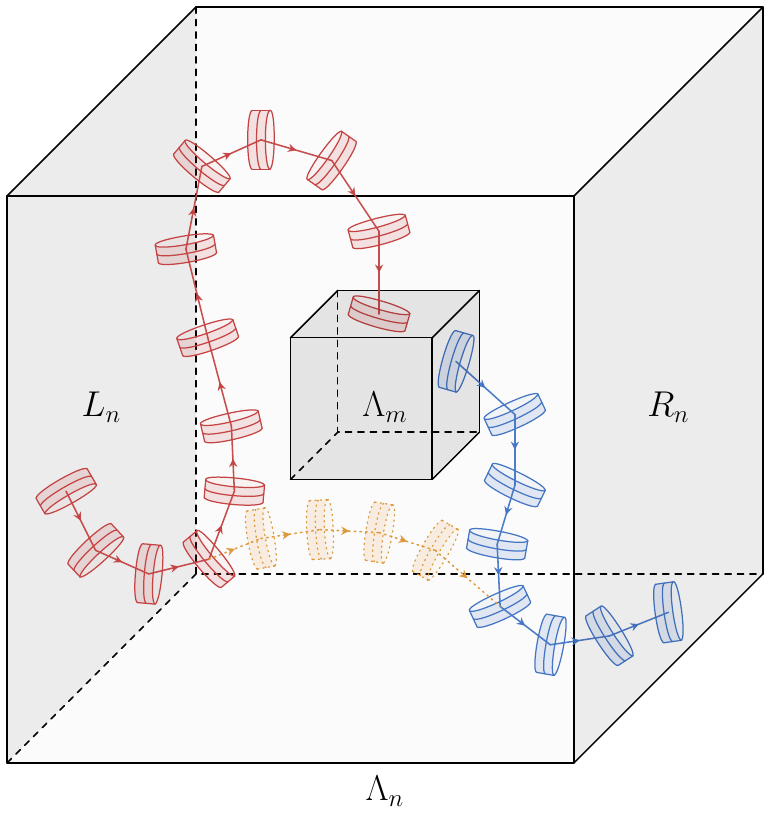}
  \caption{Illustration of two directed paths in the associated directed graph $G_{4w+13}$ for $k=2$ and $d=3$, arising on the event $\cF_{w,n,1}\cap\cF_{w,n,2}\setminus\cH_{4w+13,n}$, as used in the proof of Theorem~\ref{thm:box-crossing}(2).
  Each red, blue, or orange ``ring'' represents a vertex of $G_{4w+13}$.
  The base loops of the initial and terminal vertices of the left directed path (shown in red) intersect the left side $L_n$ of the large box $\Lm_n$ and the smaller box $\Lm_m$, respectively.
  Similarly, the base loops of the initial and terminal vertices of the right directed path (shown in blue) intersect $\Lm_m$ and the right side $R_n$ of $\Lm_n$, respectively.
  Moreover, there is no directed path connecting these two paths from any vertex of the left (red) path to any vertex of the right (blue) path, as illustrated by the middle directed path (shown in orange).}
  \label{fig:crossing}
\end{figure}

\section{Discussion}
\label{sec:discussion}
In this paper, we introduced the notion of \emph{tube percolation}, describing the escape of a loop (or, more generally, a topological sphere) to infinity in the plaquette percolation model while undergoing only bounded stretching.
We established sharp thresholds for the tubular one-arm event and the tubular box-crossing property.
Furthermore, we related the sharp phase transition for tube percolation to several other percolation thresholds.
At criticality, the behavior of tube percolation remains largely unexplored and appears to be a promising direction for further study.
We conclude by discussing several open problems and directions for future research in higher-dimensional percolation, with a particular focus on homological formulations.
\begin{enumerate}
  \item We have studied the existence of a box-crossing tube connecting two opposite faces of a large box $\Lm_n=[-n,n]^d$ in the $k$-plaquette percolation model on $\R^d$~($1\le k\le d-1$).
  A further natural object to study is a \emph{box-crossing $k$-dimensional sheet}.
  One possible way to formulate this notion is via cubical homology~(see, e.g.,~\cite[Chapter~2]{KMM2004}).
  We call a $k$-cycle in a cubical set $X$ \emph{nontrivial} if it is not the boundary of any $(k+1)$-chain in $X$.
  We say that a $k$-chain $\Gm$ in the box $\Lm_n$ (horizontally) \emph{crosses} $\Lm_n$ if its boundary $\del\Gm$ forms a nontrivial $(k-1)$-cycle in the ``horizontal boundary''
  \[
  \del_h\Lm_n:=(\del[-n,n]^k)\times[-n,n]^{d-k}
  \]
  of $\Lm_n$~(in this case, $\Gm$ represents a nontrivial generator of the relative homology group $H_k(\Lm_n,\del_h\Lm_n)$).
  An open problem is whether the existence of a box-crossing open $k$-chain~(a chain consisting of open $k$-plaquettes) exhibits a sharp threshold; that is,
  \begin{equation}\label{eq:sharp_sheet}
  \lim_{n\to\infty}\P_p(\text{$\exists$ an open $k$-chain in $\Lm_n$ that crosses $\Lm_n$})=
  \begin{cases}
  0     &\text{if $p<p_c(k,d)$,}\\
  1     &\text{if $p>p_c(k,d)$}
  \end{cases}
  \end{equation}
  for some $p_c(k,d)\in(0,1)$.
  \item The existence of a box-crossing open $1$-chain in the $1$-plaquette percolation model on $\R^d$ coincides with the usual box-crossing property in the bond percolation model.
  Hence,~\eqref{eq:sharp_sheet} holds for $k=1$ and $p_c(1,d)=p_c^{\bond}(d)$.
  Furthermore, in the $(d-1)$-plaquette percolation model on $\R^d$, the existence of a (horizontal) box-crossing open $(d-1)$-chain prevents the existence of a (vertical) box-crossing open path in the dual bond percolation model, and vice versa.
  Therefore, one may show that~\eqref{eq:sharp_sheet} also holds for $k=d-1$, and that $p_c(d-1,d)=1-p_c(1,d)$.
  More generally, since each $k$-plaquette intersects a unique $(d-k)$-plaquette on the dual-vertex set $(\Z^d)^*$ at a single point, namely, their common center, the $k$-plaquette percolation model on $\R^d$ admits a dual $(d-k)$-plaquette percolation model on $\R^d$.
  Moreover, the presence of a (horizontal) box-crossing open $k$-chain implies the absence of a (vertical) box-crossing open $(d-k)$-chain in the dual model.
  This observation suggests that, if~\eqref{eq:sharp_sheet} holds, then the duality relation
  \begin{equation}\label{eq:duality}
  p_c(k,d)+p_c(d-k,d)=1
  \end{equation}
  should hold for all $1\le k\le d-1$.
  In addition to the finite-volume characterization of the threshold $p_c(k,d)$ in~\eqref{eq:sharp_sheet}, another challenging direction is to identify an infinite $k$-dimensional object whose emergence occurs precisely at $p_c(k,d)$.
  As seen in~\eqref{eq:cycle-strict} and~\eqref{eq:tube-strict}, the duality relation~\eqref{eq:duality} cannot hold for face, core, cycle, or tube percolation, at least in sufficiently high dimension $d$.
  What would be an appropriate notion of such an infinite $k$-dimensional object, and how might it be analyzed?
  \item The final open problem is a homological generalization of the uniqueness theorem for the infinite open cluster in the bond percolation model on the hypercubic lattice $\L^d$.
  To this end, we consider infinite open $k$-chains, which are formal linear combinations of possibly infinitely many open $k$-plaquettes.
  This is, in fact, a $k$-chain in the sense of Borel--Moore homology.
  Since the hypercubic lattice is locally finite, we can likewise define the boundary operator for infinite chains and consider the notion of (non)triviality in the Borel--Moore sense.
  To distinguish this from the usual (non)triviality, we use the terminology \emph{BM-(non)trivial}.
  We now observe that the uniqueness statement ``there exists at most one infinite open cluster'' can be rephrased as follows:
  For every $0$-cycle $\gm$~(in the reduced homology sense), if $\gm$ is the boundary of an infinite open $1$-chain, then $\gm$ is trivial in the open subgraph $|\eta|$.
  In other words, every BM-trivial $0$-cycle is trivial in $|\eta|$.
  With this observation in mind, we are led to ask the following question for the $k$-plaquette percolation model on $\R^d$ and any $p\in[0,1]$:
  \[
  \P_p(\text{every BM-trivial $(k-1)$-cycle is trivial in $|\eta|$})=1?
  \]
  Here, $|\eta|$ denotes the geometric realization of a $k$-plaquette configuration $\eta\in\Om$.
\end{enumerate}

\subsection*{Statements and Declarations}
\textbf{Data Availability.}
There are no data associated with this work.\\
\textbf{Conflict of Interest.}
The authors declare that there are no conflicts of interest.

\bibliographystyle{amsplain}
\bibliography{references}

\end{document}